\documentclass[12pt]{amsart}

\usepackage{
	amsmath,
 	amsfonts,
  	amssymb,
 	amsthm,
        datetime,
	}

\usepackage{tikz-cd}
\usepackage{tikz}

\usepackage{rotating}

\usetikzlibrary{decorations.pathmorphing, decorations.pathreplacing}
 \usetikzlibrary{decorations.pathreplacing,backgrounds,decorations.markings}

\usepackage{color}
\usepackage{array}
\usepackage[english]{babel}
\usepackage[utf8]{inputenc}

\usepackage[cal=boondoxo]{mathalfa}
\usepackage{mathtools}
\usepackage{bbm}

\usepackage{hyperref}
\usepackage[capitalise]{cleveref}

\usepackage{a4wide}

\usepackage{mathabx}

\theoremstyle{plain}
\newtheorem{thm}{Theorem}[section]
\newtheorem{cor}[thm]{Corollary}
\newtheorem{lem}[thm]{Lemma}
\newtheorem{prop}[thm]{Proposition}
\newtheorem{conj}[thm]{Conjecture}

\newtheorem{thmintro}{Theorem}

\newtheorem{conjintro}{Conjecture}

\theoremstyle{definition}
\newtheorem{rem}[thm]{Remark}

\theoremstyle{definition}
\newtheorem{defn}[thm]{Definition}

\def\makeautorefname#1#2{\expandafter\def\csname#1autorefname\endcsname{#2}}
\makeautorefname{lem}{Lemma}%
\makeautorefname{prop}{Proposition}%
\makeautorefname{rem}{Remark}%
\makeautorefname{section}{Section}%

\DeclareMathOperator{\id}{Id}

\DeclareMathOperator{\End}{End}

\DeclareMathOperator{\Ind}{Ind}
\DeclareMathOperator{\Res}{Res}

\DeclareMathOperator{\GL}{GL}

\DeclareMathOperator{\im}{Im}

\DeclareMathOperator{\Irr}{Irr}
\DeclareMathOperator{\Ref}{Ref}
\DeclareMathOperator{\reg}{reg}
\DeclareMathOperator{\Gau}{Gau}
\DeclareMathOperator{\cont}{cont}

\definecolor{myblue}{rgb}{0,.5,1}
\definecolor{mygreen}{rgb}{.3,.75,.1}

\def\arxivprefix{https://arxiv.org}

\title[Calogero-Moser vs. Leclerc-Miyachi]{On a conjecture about cellular characters for the complex reflection group $G(d,1,n)$}

\author{Abel Lacabanne}
\address{Institut de Recherche en Math\'ematique et Physique\\
Universit\'e Catholique de Louvain\\ 
Chemin du Cyclotron 2\\ 
1348 Louvain-la-Neuve\\ 
Belgium}
\thanks{The author was supported by the Fonds de la Recherche Scientifique - FNRS under Grant no.~MIS-F.4536.19.}
\email{abel.lacabanne@uclouvain.be}

\begin{document}


\begin{abstract}
  We propose a conjecture relating two different sets of characters for the complex reflection group $G(d,1,n)$. From one side, the characters are afforded by Calogero-Moser cells, a conjectural generalisation of Kazhdan-Lusztig cells for a complex reflection group. From the other side, the characters arise from a level $d$ irreducible integrable representations of $\mathcal{U}_q(\mathfrak{sl}_{\infty})$. We prove this conjecture in some cases: in full generality for $G(d,1,2)$ and for generic parameters for $G(d,1,n)$.
\end{abstract}



\maketitle



Using Cherednik algebras and Calogero-Moser spaces, Bonnafé and Rouquier developed in \cite{bonnafe-rouquier} the notions of cells (right, left or two-sided) and of cellular characters of a complex reflection group $W$. Whenever this group is a Coxeter group, they conjectured in \cite[Chapter 15]{bonnafe-rouquier} that these notions coincide with the corresponding notions in the Kazhdan-Lusztig theory. As for Hecke algebras with unequal parameters, these notions heavily depend on some parameter $\mathbf{c}$ defined on the reflections of $W$ and invariant by conjugation. In this paper, we are mainly interested in the notion of cellular characters for the complex reflection group $G(d,1,n)$. If $d=1$, the group $G(d,1,n)$ is nothing else than the Weyl group of type $A_{n-1}$, and if $d=2$, we recover the Weyl group of type $B_n$. Bonnafé and Rouquier showed that if the Calogero-Moser space with parameter $\mathbf{c}$ associated to $W$ is smooth then the cellular characters are irreducible. This implies that their notion of Calogero-Moser cellular characters coincides with the notion of Kazhdan-Lusztig cellular characters in type $A$. Even in type $B$, we only have a complete description for $B_2$ \cite[Chapter 19]{bonnafe-rouquier}. For the dihedral group $G(d,d,2)$, a description of Calogero-Moser families and of Calogero-Moser cellular characters has been given by Bonnafé \cite{bonnafe-dihedral}, and these are compatible with Kazhdan-Lusztig theory.

Lusztig defined in \cite[Chapter 22]{lusztig-unequal} a notion of constructible characters of a Coxeter group, using the so-called truncated induction. He conjectures that these constructible characters are exactly the characters carried by the Kazhdan-Lusztig left cells, and proved the result in the equal parameter case. These characters surprisingly appeared in the work of Leclerc and Miyachi \cite{leclerc-miyachi}. They obtained a closed formula for canonical bases of a level $2$ irreducible integrable representation $V(\Lambda_{r_1}+\Lambda_{r_2})$ of $\mathcal{U}_q(\mathfrak{sl}_{\infty})$ (here the $\Lambda_{i}$ are the fundamental weights). By evaluating these expressions at $q=1$, Leclerc and Miyachi retrieved Lusztig's constructible characters for Weyl groups of type $B$ and $D$. With a level $d$ irreducible integral representation $V\left(\sum_{i=1}^d\Lambda_{r_i}\right)$, they defined some characters of the complex reflection group $G(d,1,n)$ in a similar manner and asked whether these characters are a good analogue of constructible characters in type $B$.

Thus, we have two sets of characters for the complex reflection group $G(d,1,n)$, namely the Calogero-Moser cellular characters and the constructible characters of Leclerc and Miyachi. Both sets of characters depend heavily on some parameters ($\mathbf{c}$ or $\mathbf{r}=(r_1,\ldots,r_d)$), and up to a suitable change of parameters, we conjecture that these two sets of characters are equal.

\begin{conjintro}[Conjecture \ref{conj:cm-lm}]
  \label{conj:conj-intro}
  Let $\mathbf{r}$ be a $d$-tuple of integers. Then there exists an explicit choice of parameter $\mathbf{c}$ for the complex reflection group $G(d,1,n)$ such that the set of Calogero-Moser $\mathbf{c}$-cellular characters and the set of Leclerc-Miyachi $\mathbf{r}$-constructible characters coincide.
\end{conjintro}

We refer to the statement in Section \ref{sec:conj} for the precise relation between the parameters $\mathbf{c}$ and $\mathbf{r}$. The main result of this paper is a proof of this conjecture in two different cases.

\begin{thmintro}[Theorem \ref{thm:proof-conj}]
  Conjecture \ref{conj:conj-intro} is true in the following two cases:
  \begin{enumerate}
  \item for $G(d,1,2)$ and any parameters,
  \item for $G(d,1,n)$ and asymptotic parameters.
  \end{enumerate}
\end{thmintro}

To support this conjecture, it would be interesting to retrieve some known properties of the Calogero-Moser cellular characters, for example the fact that for any Calogero-Moser cellular character there exists a unique irreducible constituent with minimal $b$-invariant, and its multiplicity is one. This fact is already known for constructible characters of an irreducible finite Coxeter group \cite{bonnafe-b-invariant}.

The paper is organized as follows. In the first Section, we define the first set of characters we are interested in, the Calogero-Moser cellular characters. There are several equivalent definitions of these characters in \cite{bonnafe-rouquier} and we choose to use a definition using the so-called Gaudin algebra, which is a commutative subalgebra of the group algebra of $G(d,1,n)$ over a localization of a polynomial ring. Using this definition, we give another proof of the irreducibility of Calogero-Moser cellular characters for a parameter outside of the essential hyperplanes defined by Chlouveraki. We rely on some results proven in the Appendix. In Section \ref{sec:lm_constructible}, we set up notation and define the second set of characters we are interested in, the Leclerc-Miyachi constructible characters. We show that in the asymptotic situation, these characters are irreducible. In the third section, we compute explicitly the Calogero-Moser cellular and Leclerc-Miyachi constructible characters for the complex reflection group $G(d,1,2)$. On the Calogero-Moser side, we diagonalize the action of the Gaudin algebra on representations of $G(d,1,2)$ and on the Leclerc-Miyachi side, we compute the canonical bases using the algorithm introduced by Leclerc and Toffin in \cite{leclerc-toffin}. Finally, in the last section, we state precisely the conjecture relating these characters, and show that it is valid for $G(d,1,2)$ and any choice of parameters, and for $G(d,1,n)$ for generic parameters.

\subsection*{Acknowledgments}

The author thanks C. Bonnafé for many fruitful discussions and for his guidance and G. Malle for many valuable comments on an earlier version of this paper. The author was supported by the Fonds de la Recherche Scientifique - FNRS under Grant no.~MIS-F.4536.19. 





\section{Calogero-Moser cellular characters}
\label{sec:cm_cellular}

We introduce the set of Calogero-Moser $\mathbf{c}$-cellular characters of a complex reflection group, which we define using the notion of Gaudin algebra, see \cite{bonnafe-rouquier}. In the specific case of $G(d,1,n)$, we introduce a commutative subalgebra $\mathbf{JM}_{\mathbf{c}}$ generated by the so-called Jucys-Murphy elements. Using results of the Appendix, we show that the cellular characters of $G(d,1,n)$ for the algebra $\mathbf{JM}_{\mathbf{c}}$ are sums of Calogero-Moser characters. For specific values of $\mathbf{c}$, we show that the cellular characters of $G(d,1,n)$ for the algebra $\mathbf{JM}_{\mathbf{c}}$ are irreducible, then so are the Calogero-Moser $\mathbf{c}$-cellular characters.

\subsection{Notations}

We fix $V$ a finite dimensional $\mathbb{C}$-vector space, denote by $\det\colon\GL(V)\rightarrow \mathbb{C}^*$ the determinant and by $\langle\cdot,\cdot\rangle\colon V\times V^*\rightarrow \mathbb{C}$ the duality between $V$ and the space $V^*$ of linear forms on $V$. We choose for each positive integer $d$ a $d$-th root of unity $\zeta_d$ such that $\zeta_d^{d/l}=\zeta_l$ for all $l$ dividing $d$. The group of $d$-th roots of unity will be denoted by $\mu_d$.

Let $W\subset \GL(V)$ be a finite complex reflection group. We denote by $\Ref(W)$ the set of pseudo-reflections of $W$ and for each $s\in \Ref(W)$, we choose $\alpha_s\in V^*$ and $\alpha_s^{\vee}\in V$ such that
\[
  \ker(s-\id_V) = \ker(\alpha_s)\quad\text{and}\quad\im(s-\id_V) = \mathbb{C}\alpha_s^{\vee}.
\]

We denote by $\mathcal{A}$ the set of reflecting hyperplanes of $W$ as well as by $V^{\mathrm{reg}}$ the open subset $V\setminus \bigcup_{H\in\mathcal{A}}H$. A theorem of Steinberg \cite[Theorem 4.7]{broue} shows that $V^{\reg}$ is the subset of elements of $V$ with trivial stabilizers with respect to the action of $W$.

For $H\in \mathcal{A}$, the pointwise stabilizer $W_H$ of $H$ is a cyclic group of order $e_H$ with a chosen generator $s_H\in\Ref(W)$. If $\Omega\in\mathcal{A}/W$, we denote by $e_\Omega$ the common value of $e_H$ for $H\in\Omega$. With these notations, the set of reflections of $W$ is
\[
  \Ref(W) = \left\{s_H^j\ \middle\vert\  H\in \mathcal{A},1\leq j \leq e_H-1\right\},
\]
and two reflections $s_H^j$ and $s_{H'}^{j'}$ are conjugate if and only if the hyperplanes $H$ and $H'$ are in the same orbit under the action of $W$ and $j=j'$.

We also fix $\mathbf{c}\colon \Ref(W)\rightarrow\mathbb{C},s\mapsto c_s$ a function which is invariant by conjugation. For any $H\in \mathcal{A}$ and $0\leq i \leq e_H-1$, we define
\[
  k_{H,i} = \frac{1}{e_H}\sum_{k=1}^{e_H-1}\zeta_{e_H}^{k(1-i)}c_{s_H^k},
\]
which satisfy $\sum_{i=0}^{e_H}k_{H,i} =0$; we will often consider indices modulo $e_H$ and set $k_{\Omega,i}=k_{H,i}$ for $\Omega\in\mathcal{A}/W$ and any $H\in \Omega$. We recover the function $\mathbf{c}$ via
\[
  c_{s_H^i}=\sum_{j=0}^{e_H-1}\zeta_{e_H}^{i(j-1)}k_{H,j}.
\]

\subsection{Gaudin algebra and Calogero-Moser cellular characters}

For any $y\in V$, we define an element $\mathcal{D}_y$ in the group ring of $W$ with coefficients in $\mathbb{C}[V^{\reg}]$:
\[
  \mathcal{D}_y = \sum_{s\in\Ref(W)}c_s\det(s)\frac{\langle y,\alpha_s\rangle}{\alpha_s}s.
\]
The Gaudin algebra $\Gau_{\mathbf{c}}(W)$ is the sub-$\mathbb{C}[V]$-algebra of $\mathbb{C}[V^{\reg}]W$ generated by $(\mathcal{D}_y)_{y\in V}$.

\begin{prop}[{\cite[13.4.B]{bonnafe-rouquier}}]
  The algebra $\Gau_{\mathbf{c}}(W)$ is commutative.
\end{prop}

Therefore we are in the situation of the Appendix if we set $E=\mathbb{C}W$, $A=\mathbb{C}W$ acting on $E$ by left multiplication, $P=\mathbb{C}[V^{\reg}]$ and $D_i=\mathcal{D}_{y_i}$ acting on $E$ by right multiplication, where $(y_i)_{i}$ is a basis of $V$. The following is Definition \ref{def:cellular} in our setting.

\begin{defn}
  The set of cellular characters for the algebra $\Gau_{\mathbf{c}}(W)$ is called the set of Calogero-Moser $\mathbf{c}$-cellular characters, or for short $\mathbf{c}$-cellular characters. The $\mathbf{c}$-cellular character associated to $L\in\Irr(\mathbb{C}(V)\Gau_{\mathbf{c}}(W))$ is
  \[
    \gamma_L^{\Gau_{\mathbf{c}}(W)}=\sum_{\chi\in\Irr(W)}\left[\Res_{\mathbb{C}(V)\Gau_{\mathbf{c}}(W)}^{\mathbb{C}(V)W}(\mathbb{C}(V)V_\chi)\colon L\right]\chi,
  \]
  where $V_\chi$ is a representation of $W$ affording the character $\chi$ and $[X\colon L]$ is the multiplicity of $L$ in the module $X$.
\end{defn}

\begin{rem}
If $W$ is a Coxeter group and $\mathbf{c}$ has positive values, there is a notion of $\mathbf{c}$-cellular characters arising from the Kazhdan-Lusztig theory of Hecke algebras. It is conjectured by Bonnafé and Rouquier \cite[15.2, Conjecture L]{bonnafe-rouquier} that the set of Kazhdan-Lusztig cellular characters and of Calogero-Moser characters coincide.
\end{rem}

For all $z\in V^{\reg}$, we can specialize $\mathcal{D}_y$ to an element of $\mathbb{C}W$, by evaluating the coefficents in $\mathbb{C}[V^{\reg}]$ at $z$:
\[
  \sum_{s\in\Ref(W)}c_s\det(s)\frac{\langle \alpha_s,y \rangle}{\langle\alpha_s,z \rangle}s\in\mathbb{C}W.
\]
Choosing $z=y$, one obtains a central element of $\mathbb{C}W$, called the Euler element, which does not depend on $y$:
\[
  \mathbf{eu}_{\mathbf{c}}=\sum_{s\in\Ref(W)}c_s\det(s)s.
\]

\begin{lem}
  \label{lem:action_eu}
  Let $\chi\in\Irr(W)$. For $\Omega\in\mathcal{A}/W$ and $H\in\Omega$, we define $m_{\Omega,\chi}^j=\langle\chi_{\vert W_H},\det^{-j}_{\vert W_H}\rangle_{W_H}$, where $\langle\cdot,\cdot\rangle_{W_H}$ denotes the scalar product of characters of $W_H$. The Euler element $\mathbf{eu}_{\mathbf{c}}$ acts on a representation affording the character $\chi$ by multiplication by
  \[
    \sum_{\Omega\in\mathcal{A}/W}\sum_{j=0}^{e_\Omega}\frac{\lvert\Omega\rvert e_\Omega m_{\Omega,\chi}^j}{\chi(1)}k_{\Omega,j}.
  \]
\end{lem}

\begin{proof}
  See \cite[Lemma 7.2.1]{bonnafe-rouquier}.
\end{proof}

\subsection{The imprimitive reflection group $G(d,1,n)$}

In this subsection, $V$ is of dimension $n$ with a chosen basis $(y_1,\ldots,y_n)$ with dual basis $(x_1,\ldots,x_n)$. Using this choice of basis, we identify $\GL(V)$ to $\GL_n(\mathbb{C})$. We also fix a positive integer $d$ and denote by $\zeta$ the $d$-th root of unity $\zeta_d$.

\subsubsection{The group $G(d,1,n)$ and its reflections}

It is easy to describe the group $G(d,1,n)$ in terms of matrices: it is the subgroup of $\GL(\mathbb{C})$ with elements the monomial matrices with coefficients in $\mu_d$. The permutation matrix corresponding to the transposition $(i\ j)$ will be denoted by $s_{i,j}$ and the diagonal matrix with diagonal entries $(1,\ldots,1,\zeta,1,\ldots,1)$, with $\zeta$ at the $i$-th position will be denoted by $\sigma_i$.

The set of reflections of $G(d,1,n)$ splits into $d$ conjugacy classes:
\[
  \Ref(G(d,1,n))=\bigsqcup_{k=0}^{d-1}\Ref(G(d,1,n))_k,
\]
where $\Ref(G(d,1,n))_0=\left\{\sigma_i^rs_{i,j}\sigma_i^{-r}\ \middle\vert\ 1\leq i < j \leq n, 0 \leq r \leq d-1\right\}$ and $\Ref(G(d,1,n))_k=\left\{\sigma_i^k\ \middle\vert\ 1 \leq i \leq n \right\}$ for $1 \leq k \leq d-1$.

We now give an explicit choice for $\alpha_s$ and $\alpha_s^\vee$ for any reflection $s$. For the reflection $s_{i,j,r}=\sigma_i^rs_{i,j}\sigma_i^{-r}$, the reflecting hyperplane $H_{i,j,r}$ is given by the kernel of the linear form $\alpha_{i,j,r}=x_i-\zeta^rx_j$ and the eigenspace associated to the eigenvalue $-1$ is spanned by $\alpha_{i,j,r}^\vee=\zeta^ry_i-y_j$. For the reflection $\sigma_i^k$, the reflecting hyperplane $H_{i}$ is given by the kernel of the linear form $\alpha_{i}=x_i$ and the eigenspace associated to the eigenvalue $\zeta^k$ is spanned by $\alpha_{i}^\vee=y_i$.

Under the action of $G(d,1,n)$, the set of reflecting hyperplanes $\mathcal{A}$ has only two orbits, $\Omega_0=\left\{H_{i,j,r}\ \middle\vert\ 1\leq i < j \leq n, 0 \leq r < d-1\right\}$ and $\Omega_1=\left\{H_{i}\ \middle\vert\ 1 \leq i \leq n\right\}$ which are of respective cardinal $d\frac{n(n-1)}{2}$ and $n$. 

Given a function $\mathbf{c}\colon \Ref(W)\rightarrow \mathbb{C}$ constant on the conjugacy classes, we denote its value on $\Ref(G(d,1,n))_k$ by $c_k$ and will write $k_i$ instead of $k_{\Omega_1,i}$. We will try not to introduce the parameters $k_{\Omega_0,0}$ and $k_{\Omega_0,1}$ which are respectively equal to $-\frac{c_0}{2}$ and $\frac{c_0}{2}$. Finally, the Euler element associated to $G(d,1,n)$ and $\mathbf{c}$ will be denoted by $\mathbf{eu}_{\mathbf{c},n}$.

\subsubsection{Representations and $d$-partitions}

The representation theory of $G(d,1,n)$ is well known and is governed by the $d$-partitions of $n$, see \cite[Section 5.1]{geck-jacon} for example. A partition of $n$ is a finite sequence of integers $\lambda=(\lambda_1,\ldots,\lambda_r)$ adding up to $n$ such that $\lambda_1\geq \lambda_2\geq\cdots\lambda_r >0$, and we set $\lvert \lambda \rvert = n$. A $d$-partition of $n$ is a $d$-tuple $(\lambda^{(1)},\ldots,\lambda^{(d)})$ of partitions such that $\sum_{i=1}^{d}\lvert\lambda^{(i)}\rvert = n$. The isomorphism classes of irreducible complex representations of $G(d,1,n)$ are parameterized by $d$-partitions of $n$, and for such a $d$-partition $\lambda$, we denote by $V_\lambda$ a corresponding representation.

One can describe the branching rule $G(d,1,n)\subset G(d,1,n+1)$ in terms of Young diagrams. The Young diagram $[\lambda]$ of a $d$-partition $\lambda$ of $n$ is the set
\[
  \left\{(a,b,c) \in \mathbb{Z}_{>0}\times\mathbb{Z}_{>0}\times\{1,\ldots,d\}\middle\vert 1 \leq b \leq \lambda_a^{(c)}\right\}, 
\]
whose elements will be called boxes. The content $\cont(\gamma)$ of a box $\gamma=(a,b,c)$ is the integer $b-a$. A box $\gamma$ of $[\lambda]$ is said to be removable if $[\lambda]\setminus\{\gamma\}$ is the Young diagram of a $d$-partition $\mu$ of $n-1$, and in this case, the box $\gamma$ is said to be addable to $\mu$.

\begin{prop}[{\cite[Proposition 5.1.8]{geck-jacon}}]
  Let $\lambda$ be a $d$-partition of $n$. Then
  \[
     \Ind_{G(d,1,n)}^{G(d,1,n+1)}(V_\lambda)=\bigoplus_{\mu} V_\mu
  \]
  where $\mu$ runs over the $d$-partitions of $n+1$ with Young diagram obtained by adding an addable box to the Young diagram of $\lambda$. Concerning the restriction,
   \[
     \Res_{G(d,1,n-1)}^{G(d,1,n)}(V_\lambda)=\bigoplus_{\mu} V_\mu
  \]
  where $\mu$ runs over the $d$-partitions of $n-1$ with Young diagram obtained by removing a removable box from the Young diagram of $\lambda$.
\end{prop}

Using this branching rule, we define a basis of $V_{\lambda}$ in terms of standard $d$-tableaux of shape $\lambda$, which are bijections $\mathfrak{t}\colon[\lambda]\rightarrow\{1,\ldots,n\}$ such that for all boxes $\gamma=(a,b,c)$ and $\gamma'=(a',b',c)$ we have $\mathfrak{t}(\gamma)<\mathfrak{t}(\gamma')$ if $a=a'$ and $b<b'$ or $a<a'$ and $b=b'$. Giving a standard $d$-tableau is then equivalent to giving a sequence of $d$-partitions $(\lambda^{\mathfrak{t}}[i])_{1\leq i \leq n}$ such that $[\lambda^{\mathfrak{t}}[i]]=\mathfrak{t}^{-1}(\{1,\ldots,i\})$. Therefore $V_\lambda$ is the direct sum of one dimensional spaces $D_{\mathfrak{t}}$, where for all $1\leq i \leq n$ the space $D_{\mathfrak{t}}$ is in the irreducible component $V_{\lambda^{\mathfrak{t}}[i]}$ of $\Res_{G(d,1,i)}^{G(d,1,n)}(V_\lambda)$.

\subsubsection{A commutative subalgebra of $\mathbb{C}G(d,1,n)$}

For $1\leq k \leq n$, we define the following elements of $\mathbb{C}G(d,1,n)$:
\[
  J_{k}=\mathbf{eu}_{\mathbf{c},k}-\mathbf{eu}_{\mathbf{c},k-1} = \sum_{\substack{s\in\Ref(G(d,1,k))\\s\not\in\Ref(G(d,1,k-1))}}c_s\det(s)s.
\]
If $d=1$, these elements are the usual Jucys-Murphy elements for the symmetric group $\mathfrak{S}_k$, multiplied by the scalar $c_0$.

\begin{lem}
  For all $1\leq i,j \leq n$, the Jucys-Murphy elements $J_i$ and $J_j$ commute and $J_{i+1}=s_{i,i+1,0}J_is_{i,i+1,0}-c_0\sum_{r=0}^{d-1}s_{i,i+1,r}$.
\end{lem}

\begin{proof}
  It is immediate to check that the Jucys-Murphy element $J_i$ is equal to
  \[
    J_i = -c_0\sum_{1\leq p < i}\sum_{r=0}^{d-1}s_{p,i,r} + \sum_{r=1}^{d-1}c_r\zeta^r\sigma_i^r.
  \]
  Since the conjugate of $s_{p,i,r}$ by $s_{i,i+1,0}$ is $s_{p,i+1,r}$ and the conjugate of $\sigma_i^r$ by $s_{i,i+1,0}$ is $\sigma_{i+1}^r$, we have that
  \[
    s_{i,i+1,0}J_is_{i,i+1,0}^{-1} = -c_0\sum_{1\leq p < i}\sum_{r=0}^{d-1}s_{p,i+1,r} + \sum_{r=1}^{d-1}c_r\zeta^r\sigma_{i+1}^r,
  \]
  and we obtain the induction formula.

  Since $\mathbf{eu}_{\mathbf{c},i}$ commutes with every element of $\mathbb{C}G(d,1,i)$, one see that $J_i=\mathbf{eu}_{\mathbf{c},i}-\mathbf{eu}_{\mathbf{c},i-1}$ commutes with every element of $\mathbb{C}G(d,1,i-1)$, and therefore with $J_1,\ldots,J_{i-1}$.
\end{proof}

The commutative subalgebra of $\mathbb{C}G(d,1,n)$ generated by the elements $J_1,\ldots,J_n$ is denoted by $\mathbf{JM}_{\mathbf{c}}(d,n)$. On the representation $V_\lambda$, the action of the Jucys-Murphy elements is simultaneously diagonalizable, and one can easily compute the eigenvalues using Lemma \ref{lem:action_eu}.

\begin{prop}
  Let $\lambda$ be a $d$-partition of $n$ and $\gamma=(a,b,c)$ a removable box of $[\lambda]$. Then $J_n$ acts on the component $V_{\lambda\setminus\{\gamma\}}$ of $\Res_{G(d,1,n-1)}^{G(d,1,n)}(V_\lambda)$ by multiplication by the scalar
  \[
    d(k_{1-c}-c_0(b-a)).
  \]
\end{prop}

\begin{proof}
  We start by computing the action of the Euler element on $V_{\lambda}$, and therefore we need the values of the integers $m_{\Omega,j}^{\chi_{\lambda}}$, for $\Omega\in\mathcal{A}/W$ and $0 \leq j \leq e_{e_\Omega}-1$, where $\chi_\lambda$ is the character of $V_\lambda$. From \cite[Lemma 6.1]{rouquier-schur} we have
  \[
    \frac{1}{\chi_\lambda(1)}\langle (\chi_\lambda)_{\vert_{\langle s_0\rangle}},\det\nolimits^j\rangle_{\langle s_0\rangle} = \frac{\lvert \lambda^{(j+1)}\rvert}{n}\quad\text{and}\quad\frac{1}{\chi_\lambda(1)}\langle (\chi_\lambda)_{\vert_{\langle s_{1}\rangle}},\det\rangle_{\langle s_{1}\rangle} = \frac{1}{2}-\frac{1}{n(n-1)}\sum_{\gamma\in[\lambda]}\cont(\gamma).
  \]
  Therefore, using Lemma \ref{lem:action_eu}, $\mathbf{eu}_{\mathbf{c},n}$ acts on $V_\lambda$ by multiplication by the scalar
  \[
    \omega_\lambda(\mathbf{eu}_{\mathbf{c},n}) = \sum_{j=0}^{d-1}k_{-j}\frac{dn}{\chi_\lambda(1)}\langle (\chi_\lambda)_{\vert_{\langle s_0\rangle}},\det\nolimits^j\rangle_{\langle s_0\rangle} + \frac{c_0}{2}\frac{dn(n-1)}{\chi_\lambda(1)}(\langle (\chi_\lambda)_{\vert_{\langle s_{1}\rangle}},\det\rangle_{\langle s_{1}\rangle}-\langle (\chi_\lambda)_{\vert_{\langle s_{1}\rangle}},1\rangle_{\langle s_{1}\rangle}).
  \]
  But $\langle (\chi_\lambda)_{\vert_{\langle s_{1}\rangle}},1\rangle_{\langle s_{1}\rangle}=\chi_\lambda(1)-\langle (\chi_\lambda)_{\vert_{\langle s_{1}\rangle}},\det\rangle_{\langle s_{1}\rangle}$, so that
  \[
    \omega_\lambda(\mathbf{eu}_{\mathbf{c},n}) = d\sum_{j=0}^{d-1}k_{-j}\lvert\lambda^{(j+1)}\rvert-dc_0\sum_{\gamma\in[\lambda]}\cont(\gamma)
  \]
  and we obtain the desired formula because $J_n=\mathbf{eu}_{\mathbf{c},n}-\mathbf{eu}_{\mathbf{c},n-1}$.
\end{proof}

\begin{cor}
  Let $\lambda$ be a $d$-partition of $n$, $\mathfrak{t}$ be a standard $d$-tableau of shape $\lambda$ and $1 \leq k \leq n$. The element $J_k$ acts on $D_{\mathfrak{t}}$ by multiplication by $d(k_{1-c}-c_0(b-a))$, where $\mathfrak{t}^{-1}(k)=(a,b,c)$.
\end{cor}

\subsubsection{Cellular characters for $\mathbf{JM}_{\mathbf{c}}(d,n)$ and $\mathbf{c}$-cellular characters}

As well as for the Gaudin algebra, we define cellular characters for the algebra $\mathbf{JM}_{\mathbf{c}}(d,n)$. We are again in the situation of the Appendix if we set $E=\mathbb{C}G(d,1,n)$, $A=\mathbb{C}G(d,1,n)$ acting on $E$ by left multiplication, $P=\mathbb{C}$ and $D_i=J_i$ acting on $E$ by right multiplication. The following is Definition \ref{def:cellular} in our setting.

\begin{defn}
  The $\mathbf{c}$-cellular character associated to $L\in\Irr(\mathbf{JM}_{\mathbf{c}}(d,n))$ is
  \[
    \gamma_L^{\mathbf{JM}_{\mathbf{c}}(d,n)}=\sum_{\chi\in\Irr(W)}\left[\Res_{\mathbf{JM}_{\mathbf{c}}(d,n)}^{\mathbb{C}G(d,1,n)}(V_\chi)\colon L\right]\chi,
  \]
  where $V_\chi$ is a representation of $W$ affording the character $\chi$.
\end{defn}

These characters are easier to compute than the $\mathbf{c}$-cellular characters, since $\mathbf{JM}_{\mathbf{c}}(d,n)$ is commutative and split, and are close to $\mathbf{c}$-cellular characters in the following sense.

\begin{thm}
  Every cellular character for the algebra $\mathbf{JM}_{\mathbf{c}}(d,n)$ is a sum of $\mathbf{c}$-cellular characters.
\end{thm}

\begin{proof}
  We denote by $\Gau_{\mathbf{c}}^{(0)}(d,n)$ the sub-$\mathbb{C}[V]$-algebra of $\mathbb{C}[V^{\reg}]G(d,1,n)$ generated by
  \[
    x_k\mathcal{D}_{y_k} = \sum_{r=1}^{d-1}c_r\zeta^r\sigma_k^t-c_0\sum_{r=0}^{d-1}\left(\sum_{1\leq i < k}\frac{-\zeta^rx_k}{x_i-\zeta^rx_k}s_{i,k,r}+\sum_{k < j \leq n}\frac{x_k}{x_k-\zeta^rx_j}s_{k,j,r}\right),
  \]
  for $1 \leq k \leq n$. Since $x_k$ is invertible in $\mathbb{C}(V)$ for all $k$, the algebras $\mathbb{C}(V)\Gau_{\mathbf{c}}(G(d,1,n))$ and $\mathbb{C}(V)\Gau_{\mathbf{c}}^{(0)}(d,n)$ are equal, and the $\mathbf{c}$-cellular characters are therefore equal to the cellular characters for $\Gau_{\mathbf{c}}^{(0)}(d,n)$. We then specialize the algebra $\Gau_{\mathbf{c}}^{(0)}(d,n)$ with respect to the following increasing sequence of prime ideals
  \[
    0\subset \mathfrak{p}_1\subset\mathfrak{p}_2\subset\cdots\subset\mathfrak{p}_n
  \]
  where $\mathfrak{p}_i$ is the prime ideal of $\mathbb{C}[V]$ generated by $x_1,\ldots,x_i$. We define recursively the algebra $\Gau_{\mathbf{c}}^{(i)}(d,n)$ by
  \[
    \Gau_{\mathbf{c}}^{(i)}(d,n)=\Gau_{\mathbf{c}}^{(i-1)}(d,n)/x_i\Gau_{\mathbf{c}}^{(i-1)}(d,n),
  \]
  and we denote by $\pi^{(i)}\colon\Gau_{\mathbf{c}}^{(i)}(d,n)\rightarrow\Gau_{\mathbf{c}}^{(i+1)}(d,n)$ the corresponding quotient map, and $\Pi^{(i)} = \pi^{(i-1)}\circ\cdots\pi^{(0)}$. The algebra $\Gau_{\mathbf{c}}^{(n)}(d,n)$ is hence a subalgebra of $\mathbb{C}G(d,1,n)$.
  
  By Proposition \ref{prop:sum_cell}, the cellular characters for the algebra $\Gau_{\mathbf{c}}^{(i)}(d,n)$ are sums of cellular characters for the algebra $\Gau_{\mathbf{c}}^{(i-1)}(d,n)$, and therefore the cellular characters for the algebra $\Gau_{\mathbf{c}}^{(i)}(d,n)$ are sums of $\mathbf{c}$-cellular characters.

  An easy induction shows that if $i\geq k$ then
  \[
    \Pi^{(i)}(x_k\mathcal{D}_k)=\sum_{r=1}^{d-1}c_r\zeta^r\sigma_k^r-c_0\sum_{r=0}^{d-1}\sum_{1\leq j < k}s_{j,k,r},
  \]
  and if $i<k$ then
  \[
     \Pi^{(i)}(x_k\mathcal{D}_k)=\sum_{r=1}^{d-1}c_r\zeta^r\sigma_k^r-c_0\sum_{r=0}^{d-1}\left[\sum_{j=1}^is_{j,k,r}+\sum_{i<j<k}\frac{-\zeta^rx_k}{x_j-\zeta^rx_k}s_{j,k,r}+\sum_{k < j \leq n}\frac{x_k}{x_k-\zeta^rx_j}s_{k,j,r}\right].
  \]
  The algebra $\Gau_{\mathbf{c}}^{(n)}(d,n)$ is thus equal to $\mathbf{JM}_{\mathbf{c}}(d,n)$ since $\Pi^{(n)}(x_k\mathcal{D}_k)=J_k$, which ends the proof.
\end{proof}

As a Corollary, we retrieve the fact that the $\mathbf{c}$-cellular characters are generically irreducible, see \cite[Theorem 14.4.1]{bonnafe-rouquier} or \cite[Theorem 10(3)]{bellamy}.

\begin{cor}
  \label{cor:cm-generic}
  Suppose that the parameter $\mathbf{c}$ is such that
  \[
    c_0\neq 0\quad\text{and}\quad(k_p-k_q)-c_0j\neq 0,
  \]
  for all $1\leq p\neq q \leq d$ and $-n<j<n$. Then the Calogero-Moser $\mathbf{c}$-cellular characters of $G(d,1,n)$ are irreducible.
\end{cor}

\begin{proof}
  Since $\mathbf{JM}_{\mathbf{c}}(d,n)\subset\mathbb{C}G(d,1,n)$, any simple representation occurs in some $V_\lambda$ and is therefore of the form $D_{\mathfrak{t}}$ for $\mathfrak{t}$ a standard $d$-tableau. We show that these one dimensional representations of $\mathbf{JM}_{\mathbf{c}}(n,d)$ are pairwise non-isomorphic. Let $\mathfrak{t}$ and $\mathfrak{t}'$ be two distinct standard $d$-tableaux. We prove that the sequences
  \[
    (k_{1-c_p}-c_0(b_p-a_p))_{1\leq p \leq n} \quad\text{and}\quad (k_{1-c'_p}-c_0(b'_p-a'_p))_{1\leq p \leq n}
  \]
  are different, where $(a_p,b_p,c_p) = \mathfrak{t}^{-1}(p)$ and $(a'_p,b'_p,c'_p) = (\mathfrak{t}')^{-1}(p)$. Let $1\leq p\leq n$ be the minimal integer such that $\mathfrak{t}^{-1}(p)$ and $(\mathfrak{t}')^{-1}(p)$ are different. Denote by $\mu$ the common partition $\lambda^{\mathfrak{t}}[p-1]=\lambda^{\mathfrak{t}'}[p-1]$.

  Suppose first that $-n<\cont(\mathfrak{t}^{-1}(p))-\cont((\mathfrak{t}')^{-1}(p))<n$. If $c_p\neq c'_p$ then the hypothesis on the parameter $\mathbf{c}$ implies that $k_{1-c_p}-c_0(b_p-a_p)\neq k_{1-c'_p}-c_0(b'_p-a'_p)$. If $c_p=c_p'$ then both $\mathfrak{t}^{-1}(p)$ and $(\mathfrak{t}')^{-1}(p)$ are addable boxes of $\mu^{(c_p)}$. Since there exists at most one addable box to a Young diagram with a given content, we deduce that the contents of $\mathfrak{t}^{-1}(p)$ and $(\mathfrak{t}')^{-1}(p)$ are different, and as $c_0\neq 0$ we have $k_{1-c_p}-c_0(b_p-a_p)\neq k_{1-c_p}-c_0(b'_p-a'_p)$.
  
  Therefore we may and will assume that $\lvert \cont(\mathfrak{t}^{-1}(p))-\cont((\mathfrak{t}')^{-1}(p)) \rvert \geq n$. Since the absolute value of the content of a box of a $d$-partition of $n$ cannot exceed $n-1$, the contents of $\mathfrak{t}^{-1}(p)$ and of $(\mathfrak{t}')^{-1}(p)$ are of different signs. Up to exchanging $\mathfrak{t}$ and $\mathfrak{t}'$, we suppose that the content $\mathfrak{t}^{-1}(p)$ is equal to $x>0$ and the content of $(\mathfrak{t}')^{-1}(p)$ is equal to $y<0$ (neither $x$ nor $y$ can be equal to $0$ since $0\leq x <n$, $-n < y \leq 0$ and $x-y\geq n$). The Young diagram $[\mu]$ must contain a box of content $x-1$ and a box of content $y+1$, and therefore has at least $x-y-1$ boxes. Since $x-y\geq n$, we obtain that $p-1\geq n-1$ and hence $p=n$. If $c_p\neq c'_p$ then the $d$-partition $\mu$ has at least $x-y$ boxes, which is impossible, so that $c_p=c'_p$. But $k_{1-c_p}-c_0x \neq k_{1-c_p}-c_0y$ because $c_0\neq 0$.
\end{proof}




\section{Leclerc-Miyachi constructible characters}
\label{sec:lm_constructible}

In this section, we introduce other characters of $G(d,1,n)$, whose definition was given by Leclerc and Miyachi in \cite{leclerc-miyachi}, using canonical bases of some representations over the quantum group $\mathcal{U}_q(\mathfrak{sl}_{\infty})$. If $d=2$, Leclerc and Miyachi have shown that these characters are equal to Lusztig's constructible characters \cite[Chapter 22]{lusztig-unequal} of the Coxeter group of type $B_n$, conjectured to be equal to the Kazhdan-Lusztig cellular characters. Let $q$ be an indeterminate over $\mathbb{Q}$.

\subsection{The Hopf algebra $\mathcal{U}_q(\mathfrak{sl}_{\infty})$}

The quantum group $\mathcal{U}_q(\mathfrak{sl}_{\infty})$ associated with the doubly infinite Dynkin diagram $A_{\infty}$ is the $\mathbb{Q}(q)$-algebra generated by $E_i,F_i$ and $K_i^{\pm 1}$ for $i\in\mathbb{Z}$ subject to the following relations:
\begin{align*}
  K_iK_i^{-1}=1=K_i^{-1}K_i,&\quad K_iK_j=K_jK_i,\\
  K_iE_j=q^{-\delta_{i,j-1}+2\delta_{i,j}-\delta_{i,j+1}}E_jK_i,&\quad K_iF_j=q^{\delta_{i,j-1}-2\delta_{i,j}+\delta_{i,j+1}}F_jK_i,\\
  E_iF_j-F_jE_i &= \delta_{i,j}\frac{K_i-K_i^{-1}}{q-q^{-1}},\\
  E_iE_j&=E_jE_i\text{ if }\lvert i-j\rvert >1,\\
  E_i^2E_j-(q+q^{-1})&E_iE_jE_i+E_jE_i^2=0\text{ if }\lvert i-j\rvert =1,\\
  F_iF_j&=F_jF_i\text{ if }\lvert i-j\rvert >1,\\
  F_i^2F_j-(q+q^{-1})&F_iF_jF_i+F_jF_i^2=0\text{ if }\lvert i-j\rvert =1,
\end{align*}
for all $i,j\in\mathbb{Z}$. It is a Hopf algebra, and we choose the following comultiplication $\Delta$, counit $\varepsilon$ and antipode $S$:
\begin{align*}
  \Delta(K_i)&=K_i\otimes K_i& S(K_i)&=K_i^{-1} & \varepsilon(K_i)&=1,\\
  \Delta(E_i)&=E_i\otimes 1 + K_i^{-1}\otimes E_i& S(E_i)&=-K_iE_i & \varepsilon(E_i)&=0,\\
  \Delta(E_i)&=F_i\otimes K_i + 1\otimes F_i& S(F_i)&=-F_iK_i^{-1} & \varepsilon(F_i)&=0.
\end{align*}

The fundamental roots of $\mathfrak{sl}_{\infty}$ are denoted by $(\Lambda_i)_{i\in\mathbb{Z}}$. For all $i\in\mathbb{Z}$, we denote by $V(\Lambda_i)$ the integrable irreducible representation of highest weight $\Lambda_i$. It admits $(v_\beta)_{\beta}$ as a $\mathbb{Q}(q)$-basis, where $\beta$ runs in $\left\{(\beta_k)_{k\leq i}\ \middle\vert\ \beta_k < \beta_{k+1}, \beta_k=k\text{ for } k\ll 0\right\}$. We will identify such a $\beta$ in this set with the corresponding subset $\left\{\beta_k\ \middle\vert\ k\leq i\right\}$ of $\mathbb{Z}$ and may write $j\in\beta$ or $\beta\cup\{l\}$ if $l\not\in\beta$. On this basis $(v_\beta)_\beta$, the action of the generators $E_i,F_i$ and $K_i$ are \cite{leclerc-miyachi}:
\begin{align*}
  E_iv_\beta &=
               \begin{cases}
                 v_\gamma & \text{if }i\not\in\beta\text{ and } i+1\in\beta,\text{ with } \gamma=(\beta\setminus\{i+1\})\cup\{i\},\\
                 0 & \text{otherwise},
               \end{cases}
\\
  F_iv_\beta &=
               \begin{cases}
                 v_\gamma & \text{if }i\in\beta\text{ and } i+1\not\in\beta,\text{ with } \gamma=(\beta\setminus\{i\})\cup\{i+1\},\\
                 0 & \text{otherwise},
               \end{cases}\\
  K_iv_\beta &=
               \begin{cases}
                 qv_\beta & \text{if }i\in\beta\text{ and } i+1\not\in\beta,\\
                 q^{-1}v_\beta & \text{if }i\not\in\beta\text{ and } i+1\in\beta,\\
                 v_\beta & \text{otherwise}.
               \end{cases}
\end{align*}
The highest weight vector is $v_{\beta^i}$ where $\beta^i = \mathbb{Z}_{\leq i}$.

\subsection{Fock spaces, canonical bases and constructible characters}

\subsubsection{Fock space of a representation}

Let $\mathbf{r}=(r_1,\ldots,r_d)$ be a $d$-tuple of integers with $r_1 \geq r_2\geq \cdots \geq r_d$ and we consider the fundamental weight $\Lambda_{\mathbf{r}}=\sum_{i=1}^d\Lambda_{r_i}$. The integrable irreducible $\mathcal{U}_q(\mathfrak{sl}_{\infty})$-module of highest weight $\Lambda_{\mathbf{r}}$ is denoted by $V(\Lambda_{\mathbf{r}})$. Denote by $F(\Lambda_{\mathbf{r}})=V(\Lambda_{r_1})\otimes\cdots\otimes V(\Lambda_{r_d})$ the associated Fock space, which admits $v_{\beta^{r_1}}\otimes\cdots\otimes v_{\beta^{r_d}}$ as a highest weight vector of weight $\Lambda_{\mathbf{r}}$. From now on, we view the module $V(\Lambda_r)$ inside the Fock space $F(\Lambda_{r})$. The module $F(\Lambda_{\mathbf{r}})$ has a basis $S(\Lambda_{\mathbf{r}})$ given by
\[
  S(\Lambda_{\mathbf{r}})=\left\{v_{\beta_1}\otimes\cdots\otimes v_{\beta_d}\ \middle\vert\ v_{\beta_i}\in \mathbb{Z}_{\leq r_i}\right\}.
\]
This is the standard basis of $F(\Lambda_{\mathbf{r}})$ and we prefer to write its indexing set as a set of $d$-symbols
\[
  S=
  \begin{pmatrix}
    \beta_1\\
    \beta_2\\
    \vdots\\
    \beta_d
  \end{pmatrix},
\]
where $\beta_i=(\beta_{i,k})_{k\leq r_i}$ is a sequence of integers with $\beta_{i,k}<\beta_{i,k+1}$ and $\beta_{i,k}=k$ for $k\ll 0$. The height of such a symbol is the integer $\sum_{i=1}^{d}\sum_{k\leq r_i}(\beta_{i,k}-k)$. The highest weight vector $v_{\beta^{r_1}}\otimes\cdots\otimes v_{\beta^{r_d}}$ of weight $\Lambda_{\mathbf{r}}$ corresponds to $v_{S^{0}}$, where $S^0$ is the $d$-symbol with $i$-th line equal to $\beta^{r_i}$. Finally, a $d$-symbol is said to be standard if $\beta_{i,k}<\beta_{j,k}$ for all $i\leq j$ and $k\leq r_j$.

\subsubsection{Canonical bases}

Let $x\mapsto\overline{x}$ be the involution of $\mathcal{U}_q(\mathfrak{sl}_{\infty})$ defined as the unique $\mathbb{Q}$-linear ring morphism satisfying
\begin{align*}
  \overline{q}&=q^{-1}, & \overline{K_i}&=K_i^{-1} & \overline{E_i}&=E_i, & \overline{F_i} &= F_i.
\end{align*}
Since $V(\Lambda_{\mathbf{r}})$ is a highest weight module with highest weight vector $v_{S^0}$, any element $v\in V(\Lambda_{\mathbf{r}})$ can be written $v=xv_{S^0}$, with $x\in\mathcal{U}_q(\mathfrak{sl}_{\infty})$, and we set $\overline{v}=\overline{x}v_{S^0}$.

Let $R$ be the subring of $\mathbb{Q}(q)$ of rational functions which are regular at $q=0$. Let $F_{R}(\Lambda_{\mathbf{r}})$ be the $R$-sublattice of $F(\Lambda_{\mathbf{r}})$ spanned by the standard basis $S(\Lambda_{\mathbf{r}})$.

The canonical basis $(b_{\Sigma})_{\Sigma}$ of $V(\Lambda_{\mathbf{r}})$ is indexed by the set of standard $d$-symbols and is characterized by the following properties:
\[
  b_{\Sigma} \equiv v_{\Sigma} \mod qF_{R}(\Lambda_{\mathbf{r}})\quad\text{and}\quad\overline{b_{\Sigma}}=b_{\Sigma}.
\]
Canonical bases were introduced in \cite{lusztig-canonical}, see also \cite{kashiwara}. We will denote this basis by $B(\Lambda_{\mathbf{r}})$.

\subsubsection{Constructible characters}

In \cite{leclerc-miyachi}, a closed expression of any $b_{\Sigma}\in B(\Lambda_{r})$ in the standard basis is given when $d=2$, and is compared to Lusztig's constructible characters. Leclerc and Miyachi then propose a definition of constructible characters via canonical bases for the complex reflection group $G(d,1,n)$.

To any $d$-symbol, we associate a $d$-partition $(\lambda^{(1)},\ldots,\lambda^{(d)})$ of its height by setting
\[
  \lambda_j^{(i)} = \beta_{i,r_i-j+1}-(r_i-j+1).
\]
This is a bijection between the set of $d$-partitions of $n$ and the set of $d$-symbols of height $n$.  

\begin{defn}[{\cite[6.3]{leclerc-miyachi}}]
  For a standard $d$-symbol $\Sigma$ of height $n$, we write the expression of $b_{\Sigma}$ in terms of the standard basis
  \[
    b_{\Sigma}=\sum_{S}a_{\Sigma}^S(q)v_S,
  \]
  with $S$ running in the set of $d$-partitions of $n$. The Leclerc-Miyachi $\mathbf{r}$-constructible character of $G(d,1,n)$ corresponding to the standard $d$-symbol $\Sigma$ is
  \[
    \gamma_{\Sigma}=\sum_{S}a_{\Sigma}^S(1)\chi_S,
  \]
  where $\chi_{S}$ is the character of the representation of $G(d,1,n)$ associated with the partition corresponding to the $d$-symbol $S$.
\end{defn}
 
\subsection{The asymptotic case}

Since we aim to compare the set of Calogero-Moser cellular characters and the set of Leclerc-Miyachi constructible character, we expect that the Leclerc-Miyachi constructible characters enjoy a generic property similar to Corollary \ref{cor:cm-generic}. This generic property on the parameter $c$ will turn out to be an asymptotic property on the parameter $\mathbf{r}$.

\begin{lem} 
  Let $\mathbf{r}=(r_1,\ldots,r_d)$ and $n\in\mathbb{N}$. We suppose that $r_{i}-r_{i+1}\geq n$ for all $1 \leq i \leq d-1$. Then every $d$-symbol of height at most $n$ is standard.
\end{lem}

\begin{proof}
  Let $S=(\beta_i)_{1\leq i \leq d}$ be a symbol of height smaller than $n$. By the hypothesis on the parameters, we necessarily have $\beta_{i,k}=k$ for $k\leq r_{i+1}$. Therefore if $i<j$ and $k\leq r_j$ we have $\beta_{i,k} = k \leq \beta_{j,k}$ and the $d$-symbol $S$ is standard.
\end{proof}

If $r_{i}-r_{i+1}\geq n$ then the number of $\mathbf{r}$-constructible characters of $G(d,1,n)$ is the same as the number of irreducible characters of $G(d,1,n)$. It remains to show that these $\mathbf{r}$-constructible characters are irreducible.

\begin{thm}
  \label{thm:lm-asymptotic}
  Let $\mathbf{r}=(r_1,\ldots,r_d)$ and $k\in\mathbb{N}$. We suppose that $r_{i}-r_{i+1}\geq n$ for all $1 \leq i \leq d-1$. For any $d$-symbol $\Sigma$ of height at most $n$ we have $b_{\Sigma}=v_{\Sigma}$.
\end{thm}

\begin{proof}
  This is an application of the algorithm presented in \cite{leclerc-toffin} for the computation of the canonical basis. We show that the intermediate basis $(A_{\Sigma})_{\Sigma}$ of \cite[Section 4.1]{leclerc-toffin} satisfies $A_\Sigma=v_\Sigma$, which implies that $b_{\Sigma}=v_{\Sigma}$ since $\overline{A_\Sigma}=A_\Sigma$. We proceed by induction on $n$.

  For $n=1$, any $d$-symbol of height $1$ is given by $S_l=(\beta_i)_{1\leq i \leq d}$ with $\beta_{i,k}=k$ for every $1\leq i \leq d$ and $k\leq r_{i}$ except for $\beta_{l,r_l}=r_{l}+1$. We immediately obtain that $A_{S_l} = F_{r_l}v_{S^0}$. By the hypothesis on $\mathbf{r}$, the only line $\beta_k$ of $\Sigma$ with $r_l\in\beta_k$ and $r_l+1\not\in\beta_k$ is $\beta_l$. Therefore $F_{r_l}v_{S^0}=v_{S_l}$.

  Suppose that for all parameters $\mathbf{r}$ such that $r_{i}-r_{i+1}\geq n$ for all $1\leq i \leq d$ we have $A_{\Sigma}=v_{\Sigma}$ for all standard $d$-symbols $\Sigma$ of height $n$. Let $\mathbf{r}$ be a parameter such that $r_{i}-r_{i+1}\geq n+1$ for all $1\leq i \leq d$ and $\Sigma=(\beta_i)_{1\leq i \leq d}$ be a standard $d$-symbol of height $n+1$. Let $i_0$ be the greatest integer such that $\beta_{i_0}\neq \beta^{r_{i_0}}$ and $k_0$ the smallest integer such that $\beta_{i_0,k_0}>k_0$. We write $\beta_{i_0,k_0}=k_0'+1$ with $k_0'\geq k_0$. Since the height of $\Sigma$ is $n+1$, we have $\beta_{i_0,k_0}-k_0\leq n+1$ so that $k_0'\leq n+k_0$. In order to apply the algorithm of Leclerc-Toffin, one must find the smallest integer $k$ such that there exists $1 \leq i \leq d$ and $l\leq k$ with $\beta_{i,l}=k+1$. Let us show that this integer is $k_0'$.
  
  Fix $k<k_0'$, $1\leq i \leq d$ and $l\leq k$. Suppose first that $i>i_0$. By definition of $i_0$, we have $\beta_{i,l}=l\neq k+1$. Now suppose that $i=i_0$. If $l<k_0$ then by definition of $k_0$ we have $\beta_{i_0,l}=l\neq k+1$. If $l\geq k_0$ then $\beta_{i_0,l}\geq \beta_{i_0,k_0} = k_0'+1 > k+1$. Finally, suppose that $i<i_0$. Since $l\leq k$, we have $l\leq k_0'+1 \leq k_0+n+1$. Since $\Sigma$ is of height $n+1$, if $l \leq r_i-(n+1-(\beta_{i_0,k_0}-k_0))$ we have $\beta_{i,l}=l$. But $r_{i}-(n+1) \geq r_{i+1} \geq r_{i_0}$ and therefore $r_i-(n+1-(\beta_{i_0,k_0}-k_0))\geq r_{i_0} + k_0' + 1 - k_0$. As obviously $k_0\leq r_{i_0}$ we obtain that $r_i-(n+1-(\beta_{i_0,k_0}-k_0))\geq k_0'+1\geq k$. Hence if $l\leq k$ then $\beta_{i,l}=l\neq k+1$.

 Therefore we obtain
  \[
    A_\Sigma = F_{k_0'}A_{\Sigma'},
  \]
  where $\Sigma'$ is the standard $d$-symbol obtained from $\Sigma$ by replacing only $\beta_{i_0,k_0}$ by $k_0'$. Then $\Sigma'$ is of height $n$, and the induction hypothesis shows that $A_{\Sigma'}=v_{\Sigma'}$.

  In order to conclude, it remains to show that $F_{k_0'}v_{\Sigma'}=v_{\Sigma}$. If $i>i_0$ then $\beta_i=\beta^{r_i}$ and since $k_0'\geq k_0> r_i$ neither $k_0$ nor $k_0'$ appear in $\beta_i$. If $i<i_0$, we have already shown that if $l \leq r_i-(n+1-(\beta_{i_0,k_0}-k_0))$ we have $\beta_{i,l}=l$ and that $r_i-(n+1-(\beta_{i_0,k_0}-k_0))\geq k_0'+1$ so that $\beta_{i,k_0'+1}=k_0'+1$ and $\beta_{i,k_0'}=k_0'$ and both $k_0'$ and $k_0'+1$ appear in $\beta_i$. Hence, from the definition of the action of $F_{k_0'}$ via the comultiplication, we find that $F_{k_0'}v_{\Sigma'}=v_{\Sigma}$.
\end{proof}

The following corollary translates Theorem \ref{thm:lm-asymptotic} in terms of Leclerc-Miyachi constructible characters for $G(d,1,n)$.

\begin{cor}
  \label{cor:lm-asymptotic}
  Let $\mathbf{r}=(r_1,\ldots,r_d)$ and $k\in\mathbb{N}$ and suppose that $r_{i}-r_{i+1}\geq n$ for all $1 \leq i \leq d-1$. Then the Leclerc-Miyachi $\mathbf{r}$-constructible characters for the group $G(d,1,n)$ are the irreducible characters.
\end{cor}




\section{Computations for $G(d,1,2)$ and comparison}

In this section, we compute explicitly the set of $\mathbf{c}$-cellular characters for the group $G(d,1,2)$ for any choice of parameter $\mathbf{c}$. We also compute explicitly the Leclerc-Miyachi $\mathbf{r}$-constructible characters for any choice of parameter $\mathbf{r}$.

\subsection{On the Calogero-Moser side}

We will freely use the notations of Section \ref{sec:cm_cellular}, but simplify them in the special case of $G(d,1,2)$. For simplicity, we prefer to denote by $(x,y)$ the standard basis of $\mathbb{C}^2$ and by $(X,Y)$ its dual basis. Let $s$ be the reflection denoted by $s_{1,2,0}$ and $t$ be the reflection denoted by $\sigma_1$, so that $G(d,1,2)$ has the following presentation
\[
  \left\langle s,t\ \middle\vert\ s^2=1,t^d=1,stst=tsts\right\rangle.
\]
We also denote by $s_k$ the reflection $s_{1,2,k}$. There are $2d+\binom{d}{2}$ irreducible representations of $G(d,1,2)$, namely $\eta_i,\eta_i'$ of dimension $1$ for $1\leq i \leq d$ and $\rho_{i,j}$ of dimension $2$ for $1 \leq i < j \leq d$. Their respective characters are denoted by $\xi_i,\xi_i'$ and $\chi_{i,j}$ and the values of the representations on the generators are given in Table \ref{tbl:rep-d12}
\begin{table}[h!]
\centering
\begin{tabular}{c!{\vrule width 1.5pt}c|c}
                               & $s$ & $t$ \\\noalign{\hrule height 1.5pt}
$\eta_i, 1\leq i \leq d$       & $1$ & $\zeta^{i-1}$ \\\hline 
$\eta'_i, 1\leq i \leq d$      & $-1$ & $\zeta^{i-1}$\\\hline 
$\rho_{i,j}, 1\leq i<j \leq d$  & $\begin{pmatrix}0&1\\1&0\end{pmatrix}$ & $\begin{pmatrix}\zeta^{i-j}&0\\0&\zeta^{j-i}\end{pmatrix}$
\end{tabular}
 \caption{Action of $s$ and $t$ on irreducibles repesentations of $G(d,1,2)$}\label{tbl:rep-d12}
\end{table}

Finally we choose a parameter $\mathbf{c}\colon\Ref(G(d,1,2))\rightarrow \mathbb{C}$, define $k^{\#}_i=k_{1-i}$ and we again set $\zeta=\zeta_d$.

The Gaudin algebra $\Gau_{\mathbf{c}}$ over $\mathbb{C}[X,Y]$ is generated by the following two elements
\[
  \mathcal{D}_x=\sum_{k=0}^{d-1}c_k\zeta^k\frac{1}{X}\sigma_1^k - c_0\sum_{k=0}^{d-1}\frac{1}{X-\zeta^{k}Y}s_k\quad\text{and}\quad\mathcal{D}_y= \sum_{k=0}^{d-1}c_k\zeta^k\frac{1}{Y}\sigma_2^k + c_0\sum_{k=0}^{d-1}\frac{\zeta^k}{X-\zeta^{k}Y}s_k.
\]
Since $\mathbb{C}(X,Y)\Gau_{\mathbf{c}}\subset \mathbb{C}(V)G(d,1,2)$, every irreducible $\mathbb{C}(X,Y)\Gau_{\mathbf{c}}$-module appears in the restriction of an irreducible representation of $G(d,1,2)$ over $\mathbb{C}(X,Y)$. We then denote by $\mathcal{L}_i$ (resp. $\mathcal{L}'_i$, resp. $\mathcal{L}_{i,j}$) the restriction of $\mathbb{C}(V)\eta_i$ (resp. $\mathbb{C}(V)\eta'_i$, resp. $\mathbb{C}(V)\rho_{i,j}$) to $\mathbb{C}(X,Y)\Gau_{\mathbf{c}}$. The following easy lemma will be useful in the computations.

\begin{lem}
  \label{lem:rat_frac}
  In $\mathbb{C}(V)=\mathbb{C}(X,Y)$, for every $1\leq l \leq d$ we have
  \begin{equation}
  \sum_{k=0}^{d-1}\frac{\zeta^{kl}}{X-\zeta^{k}Y} = \frac{dX^{l-1}Y^{d-l}}{X^d-Y^d}.\label{eq:identity-frac}
\end{equation}
\end{lem}

We also give two other generators of $\mathbb{C}(X,Y)\Gau_{\mathbf{c}}$, which differ from $\mathcal{D}_x$ and $\mathcal{D}_y$ by multiplication by a scalar:
\[
  \mathcal{D}'_x = \frac{X(X^d-Y^d)}{d}\mathcal{D}_x\quad\text{and}\quad\mathcal{D}'_y=\frac{Y(X^d-Y^d)}{d}\mathcal{D}_y.
\]

\begin{lem}
  \label{lem:restriction-gaudin}
  The actions of $\mathcal{D}'_x$ and $\mathcal{D}'_y$ on the restrictions of irreducible representations of $G(d,1,2)$ are given in the Table  \ref{tbl:action-gau}.
  \begin{table}[h!]
    \centering
    \begin{tabular}[h!]{c!{\vrule width 1.5pt}c|c}
        & $\mathcal{D}'_x$&$\mathcal{D}'_y$\\\noalign{\hrule height 1.5pt}
      $\mathcal{L}_i,\ 1\leq i \leq d$  & $(X^d-Y^d)k^{\#}_{i}-c_0X^d$ & $(X^d-Y^d)k^{\#}_{i}+c_0Y^d$\\\hline
      $\mathcal{L}'_i,\ 1\leq i \leq d$  & $(X^d-Y^d)k^{\#}_{i}+c_0X^d$ & $(X^d-Y^d)k^{\#}_{i}-c_0Y^d $\\\hline
      $\mathcal{L}_{i,j},\ 1\leq i<j\leq d$ & $\left(\begin{smallmatrix} (X^d-Y^d)k^{\#}_{i} & -c_0X^{d-(j-i)}Y^{j-i} \\ -c_0 X^{j-i}Y^{d-(j-i)} & (X^d-Y^d)k^{\#}_{j}\end{smallmatrix}\right)$ & $\left(\begin{smallmatrix} (X^d-Y^d)k^{\#}_{j} & c_0X^{d-(j-i)}Y^{j-i} \\ c_0 X^{j-i}Y^{d-(j-i)} & (X^d-Y^d)k^{\#}_{i}\end{smallmatrix}\right)$
    \end{tabular}
  \caption{Actions of $\mathcal{D}'_x$ and $\mathcal{D}'_y$}\label{tbl:action-gau}
  \end{table}
\end{lem}

\begin{proof}
  Lets us start with the action of $\mathcal{D}'_x$ on $\mathcal{L}_i$. It is given by
\[
  \eta_i(\mathcal{D}'_x) = \frac{X(X^d-Y^d)}{d}\left(\sum_{r=1}^d\frac{c_t\zeta^{ri}}{X}-c_0\sum_{k=0}^{d-1}\frac{1}{X-\zeta^kY}\right) = (X^d-Y^d) k^{\#}_{i} - c_0X^d,
\]
the last equality following from the definition $k^{\#}_{i}$ and from Lemma \ref{lem:rat_frac}. 

Similar computations can be made for the action of $\mathcal{D}'_y$, and for the representation $\mathcal{L}_i'$.

Let $1 \leq i < j \leq n$ and we compute the action of $\mathcal{D}'_x$ on $\mathcal{L}_{i,j}$ :
\begin{align*}
  \rho_{i,j}(\mathcal{D}'_x) &= \frac{X^d-Y^d}{d}\sum_{r=1}^{d-1} c_r\begin{pmatrix}\zeta^{ri}&0\\0&\zeta^{rj}\end{pmatrix} - \frac{X(X^d-Y^d)}{d}c_0\sum_{k=0}^{d-1}\frac{1}{X-\zeta^kY}\begin{pmatrix}0& \zeta^{k(d-(j-i))} \\ \zeta^{k(j-i)} & 0\end{pmatrix}\\
&=\begin{pmatrix} (X^d-Y^d)k^{\#}_{i} & -c_0X^{d-(j-i)}Y^{j-i} \\ -c_0 X^{j-i}Y^{d-(j-i)} & (X^d-Y^d)k^{\#}_{j}\end{pmatrix},
\end{align*}
using again the definition of $k^{\#}_{i}$ and $k^{\#}_{j}$ and Lemma \ref{lem:rat_frac}. The action of $\mathcal{D}'_y$ is obtained by a similar argument.
\end{proof}

The $2$-dimensional representations $\mathcal{L}_{i,j}$ have different behaviour depending on the parameter $\mathbf{c}$.

\subsubsection{When $c_0=0$}

In this subsection only, we suppose that $c_0=0$. The matrices giving the action of $\mathcal{D}_x'$ and $\mathcal{D}_y'$ are all diagonal and we readily see that $\mathcal{L}_i\simeq \mathcal{L}'_i$ and that $\mathcal{L}_{i,j}\simeq\tilde{\mathcal{L}}_{i,j}\oplus \tilde{\mathcal{L}}_{j,i}$, where $\tilde{\mathcal{L}}_{i,j}$ is the $1$-dimensional representation where $\mathcal{D}'_x$ acts by $(X^d-Y^d)k^{\#}_{i}$ and $\mathcal{D}'_y$ by $(X^d-Y^d)k^{\#}_{j}$. Notice that with this notation, the module $\tilde{\mathcal{L}}_{i,i}$ is nothing else than $\mathcal{L}_i$.

Moreover, we have an isomorphism between $\tilde{\mathcal{L}}_{i,j}$ and $\tilde{\mathcal{L}}_{p,q}$ if and only if $k^{\#}_{i}=k^{\#}_{p}$ and $k^{\#}_{j}=k^{\#}_{q}$. We therefore define an equivalence relation on the set $\{1,\ldots,d\}$ by $i\sim j$ if and only if $k^{\#}_{i}=k^{\#}_{j}$. Simple $\mathbb{C}(X,Y)\Gau_{\mathbf{c}}$-modules are then parameterized by pairs of equivalence class for $\sim$: a representative of the class $\mathcal{L}_{\mathcal{O},\mathcal{O}'}$ labeled by $\mathcal{O}$ and $\mathcal{O}'$ is $\tilde{\mathcal{L}}_{i,j}$, where $i\in\mathcal{O}$ and $j\in\mathcal{O}'$.

From the above description of restrictions of representations of $G(d,1,2)$ to $\mathbb{C}(X,Y)\Gau_{\mathbf{c}}$, we obtain:

\begin{prop}
  Let $\mathcal{O}$ and $\mathcal{O'}$ be two equivalence classes for $\sim$.
  
  The $\mathbf{c}$-cellular character corresponding to the class $\mathcal{L}_{\mathcal{O},\mathcal{O}}$ is
  \[
    \gamma_{\mathcal{O},\mathcal{O}}=\sum_{i\in\mathcal{O}}(\chi_i+\chi'_i)+\sum_{\substack{i,j\in \mathcal{O}\\i<j}}2\chi_{i,j}.
  \]

  The $\mathbf{c}$-cellular character corresponding to the class $\mathcal{L}_{\mathcal{O},\mathcal{O}'}$ is
  \[
    \gamma_{\mathcal{O},\mathcal{O}'}=\sum_{\substack{i\in\mathcal{O},j\in \mathcal{O}'\\i<j}}\chi_{i,j}+\sum_{\substack{i\in\mathcal{O},j\in \mathcal{O}'\\i>j}}\chi_{j,i}.
  \]

  By definition, every $\mathbf{c}$-cellular character is equal to one of those above.
\end{prop}

\subsubsection{When $c_0\neq 0$}
\label{sec:c0<>0}

In this subsection only, we suppose that $c_0\neq 0$. The matrices giving the action of $\mathcal{D}'_x$ and $\mathcal{D}'_y$ on the representation $\mathcal{L}_{i,j}$ are not diagonal, but these representations can still have an invariant one-dimensional subspace.

\begin{lem}
  Depending on the values of the parameter $\mathbf{c}$, the representation $\mathcal{L}_{i,j}$ is reducible if and only if one of the following is true:
  \begin{itemize}
    \item if $k^{\#}_i=k^{\#}_j$ and $d$ is even then $\mathcal{L}_{i,j}$ is isomorphic to $\mathcal{L}_{i,j}^+\oplus \mathcal{L}_{i,j}^-$, where $\mathcal{L}_{i,j}^+$ and $\mathcal{L}_{i,j}^-$ are two non-isomorphic one-dimensional representations, which are not isomorphic to some $\mathcal{L}_k$ or $\mathcal{L}'_k$,
  \item if $k^{\#}_i-k^{\#}_j=c_0$ then $\mathcal{L}_{i,j}\simeq \mathcal{L}_i \oplus \mathcal{L}_j'$,
  \item if $k^{\#}_i-k^{\#}_j=-c_0$ then $\mathcal{L}_{i,j}\simeq \mathcal{L}_i' \oplus \mathcal{L}_j$.
  \end{itemize}
\end{lem}

\begin{proof}
  We diagonalize the matrices $\rho_{i,j}(\mathcal{D}'_x)$ and $\rho_{i,j}(\mathcal{D}'_y)$. Note that these two matrices have the same trace and determinant, and therefore the same characteristic polynomial equal to
  \[
    \mathbf{t}^2-(X^d-Y^d)(k^{\#}_i+k^{\#}_j)\mathbf{t}+(X^d-Y^d)^2k^{\#}_ik^{\#}_j-c_0^2X^dY^d.
  \]
  This polynomial is split in $\mathbb{C}(X,Y)$ if and only if its discriminant
  \begin{multline*}
      (X^d-Y^d)^2(k^{\#}_i+k^{\#}_j)^2-4((X^d-Y^d)^2k^{\#}_ik^{\#}_j-c_0^2X^dY^d) =\\
    (k^{\#}_i-k^{\#}_j)^2X^{2d}+2(2c_0^2-(k^{\#}_i-k^{\#}_j)^2)X^dY^d+(k_i-k_j)Y^{2d}
  \end{multline*}
  is a square in $\mathbb{C}(X,Y)$. This homogeneous polynomial is then a square in $\mathbb{C}(X,Y)$ if and only if $d$ is even and $k^{\#}_i=k^{\#}_j$ or $c_0^2=(k^{\#}_i-k^{\#}_j)^2$.

  Suppose first that $k^{\#}_i=k^{\#}_j$ and $d$ is even. One check that the vectors
  \[
    \begin{pmatrix}
      X^{d/2-(j-i)}\\
      -Y^{d/2-(j-i)}
    \end{pmatrix}
    \quad\text{and}\quad
    \begin{pmatrix}
      X^{d/2-(j-i)}\\
      Y^{d/2-(j-i)}
    \end{pmatrix}
  \]
  are common eigenvectors for $\rho_{i,j}(\mathcal{D}'_x)$ and $\rho_{i,j}(\mathcal{D}'_y)$ with respective eigenvalues
  \[
    (X^d-Y^d)k^{\#}_i+c_0X^{d/2}Y^{d/2}\quad\text{and}\quad (X^d-Y^d)k^{\#}_i-c_0X^{d/2}Y^{d/2}
  \]
  for $\rho_{i,j}(\mathcal{D}'_x)$ and
  \[
    (X^d-Y^d)k^{\#}_i-c_0X^{d/2}Y^{d/2}\quad\text{and}\quad (X^d-Y^d)k^{\#}_i+c_0X^{d/2}Y^{d/2}
  \]
  for $\rho_{i,j}(\mathcal{D}'_y)$. This shows the first assertion of the lemma.

  Now, suppose that $k^{\#}_i-k^{\#}_j=c_0$. We check that
  \[
    \begin{pmatrix}
      Y^{j-i}\\
      X^{j-i}
    \end{pmatrix}
    \quad\text{and}\quad
    \begin{pmatrix}
      X^{d-(j-i)}\\
      -Y^{d-(j-i)}
    \end{pmatrix}
  \]
  are common eigenvectors for $\rho_{i,j}(\mathcal{D}'_x)$ and $\rho_{i,j}(\mathcal{D}'_y)$ with respective eigenvalues
  \[
    k_j^{\#}X^d-k_i^{\#}Y^d\quad\text{and}\quad k_i^{\#}X^d-k_j^{\#}Y^d
  \]
  for $\rho_{i,j}(\mathcal{D}'_x)$ and
  \[
    k_i^{\#}X^d-k_j^{\#}Y^d\quad\text{and}\quad k_j^{\#}X^d-k_i^{\#}Y^d
  \]
  for $\rho_{i,j}(\mathcal{D}'_y)$. Moreover, using the equality $c_0=k^{\#}_i-k^{\#}_j$, it is easy to see that this gives an isomorphism $\mathcal{L}_{i,j}\simeq \mathcal{L}_i\oplus \mathcal{L}'_j$.

  If $k^{\#}_i-k^{\#}_j=-c_0$, a similar argument shows that $\mathcal{L}_{i,j}\simeq \mathcal{L}'_i\oplus \mathcal{L}_j$, which ends the proof of the lemma.
\end{proof}

We now have a complete description of simple $\mathbb{C}(X,Y)\Gau_{\mathbf{c}}$-modules, and of the restrictions of the representations of $G(d,1,2)$ to $\mathbb{C}(X,Y)\Gau_{\mathbf{c}}$. The only isomorphism between simple modules are the following:
\begin{itemize}
\item if $k_i^{\#}=k_j^{\#}$ then $\mathcal{L}_i\simeq \mathcal{L}_j$ and $\mathcal{L}'_i\simeq \mathcal{L}'_j$,
\item if $c_0^2 \neq (k_i^{\#}-k_j^{\#})^2$, $c_0^2 \neq (k_p^{\#}-k_q^{\#})^2$, $k_i^{\#}=k_p^{\#}$ and $k_j^{\#}=k_q^{\#}$ then $\mathcal{L}_{i,j}\simeq \mathcal{L}_{p,q}$,
\item if $d$ is even and $k_i^{\#}=k_j^{\#}=k_p^{\#}=k_q^{\#}$ then $\mathcal{L}^+_{i,j}\simeq \mathcal{L}^+_{p,q}$ and $\mathcal{L}^-_{i,j}\simeq \mathcal{L}^-_{p,q}$.
\end{itemize}

We again define an equivalence relation on the set $\{1,\ldots,d\}$ by $i\sim j$ if and only if $k^{\#}_{i}=k^{\#}_{j}$. We can parameterize the classes of simple modules using the equivalence classes of $\sim$:
\begin{itemize}
\item for $\mathcal{O}$ an equivalence class, we have two classes $\mathcal{L}_{\mathcal{O}}$ and $\mathcal{L}_{\mathcal{O}'}$ with a respective representative $\mathcal{L}_i$ and $\mathcal{L}'_i$ for $i\in\mathcal{O}$.
\item for $\mathcal{O}$ and $\mathcal{O}'$ two equivalence classes such that $c_0^2\neq (k_{i}^{\#}-k_j^{\#})^2$ for any $i\in\mathcal{O}$ and $j\in\mathcal{O}'$, the classification depends moreover on the parity of $d$:
  \begin{enumerate}
  \item if $d$ is odd, we have one class of simple modules $\mathcal{L}_{\mathcal{O},\mathcal{O'}}$ with representative $\mathcal{L}_{i,j}$ with $i\in\mathcal{O}$, $j\in\mathcal{O}'$ and  $i<j$ (or $\mathcal{L}_{j,i}$ with $i\in\mathcal{O}$, $j\in\mathcal{O}'$ and $j<i$),
  \item if $d$ is even and $\mathcal{O}\neq \mathcal{O}'$, we have one class $\mathcal{L}_{\mathcal{O},\mathcal{O}'}$ with representative $\mathcal{L}_{i,j}$ with $i\in\mathcal{O}$, $j\in\mathcal{O}'$ and  $i<j$ (or $\mathcal{L}_{j,i}$ with $i\in\mathcal{O}$, $j\in\mathcal{O}'$ and $j<i$),
  \item if $d$ is even, $\mathcal{O}=\mathcal{O}'$ and $\lvert \mathcal{O}\rvert \geq 2$, we have two classes $\mathcal{L}_{\mathcal{O},\mathcal{O}'}^+$ and $\mathcal{L}_{\mathcal{O},\mathcal{O}'}^-$ with respective representatives $\mathcal{L}_{i,j}^+$ and $\mathcal{L}_{i,j}^-$ for $i,j\in\mathcal{O}$ and $i<j$.
  \end{enumerate}
\end{itemize}

Note that the classes $\mathcal{L}_{\mathcal{O},\mathcal{O}}$, $\mathcal{L}^+_{\mathcal{O},\mathcal{O}}$ and $\mathcal{L}^-_{\mathcal{O},\mathcal{O}}$ only exist if $\lvert\mathcal{O}\rvert \geq 2$. From the above description of restrictions of representations of $G(d,1,2)$ to $\mathbb{C}(X,Y)\Gau_{\mathbf{c}}$, we obtain:

\begin{prop}
  \label{prop:cm-cell-d12-c0<>0}
  Let $\mathcal{O}$ and $\mathcal{O'}$ be two equivalence classes for $\sim$.

  The $\mathbf{c}$-cellular character corresponding to the class $\mathcal{L}_{\mathcal{O}}$ is
  \[
    \gamma_{\mathcal{O}} = \sum_{i\in\mathcal{O}}\left(\chi_i+\sum_{j\in\mathcal{O}',i<j}\chi_{i,j}+\sum_{j\in\mathcal{O'},j<i}\chi_{j,i}\right),
  \]
  where $\mathcal{O}'$ is an equivalence class for $\sim$ such that $k_i^{\#}-k_j^{\#}=c_0$ (if such a class exists, it is unique).
  
  The $\mathbf{c}$-cellular character corresponding to the class $\mathcal{L}'_{\mathcal{O}}$ is
  \[
    \gamma'_{\mathcal{O}} = \sum_{i\in\mathcal{O}}\left(\chi'_i+\sum_{j\in\mathcal{O}',i<j}\chi_{i,j}+\sum_{j\in\mathcal{O'},j<i}\chi_{j,i}\right),
  \]
  where $\mathcal{O}'$ is an equivalence class for $\sim$ such that $k_i^{\#}-k_j^{\#}=-c_0$ (if such a class exists, it is unique).
  
  If $c_0^2\neq (k_{i}^{\#}-k_j^{\#})^2$, the $\mathbf{c}$-cellular character corresponding to the class $\mathcal{L}_{\mathcal{O},\mathcal{O}'}$ is
  \[
    \gamma_{\mathcal{O},\mathcal{O}'}=\sum_{\substack{i\in\mathcal{O},j\in\mathcal{O}'\\i<j}}\chi_{i,j}+\sum_{\substack{i\in\mathcal{O},j\in\mathcal{O}'\\j<i}}\chi_{j,i}.
  \]
  
  The $\mathbf{c}$-cellular character corresponding to the classes $\mathcal{L}_{\mathcal{O},\mathcal{O}}$, $\mathcal{L}^+_{\mathcal{O},\mathcal{O}}$ and $\mathcal{L}^-_{\mathcal{O},\mathcal{O}}$ is
  \[
    \gamma_{\mathcal{O},\mathcal{O}}=\sum_{\substack{i,j\in\mathcal{O}\\i<j}}\chi_{i,j}.
  \]

   By definition, every $\mathbf{c}$-cellular character is equal to one of those above.
\end{prop}

\subsection{Vectors of height $2$ of the canonical basis}
\label{sec:height2}

Now, we turn to the Leclerc-Miyachi constructible characters for $G(d,1,2)$ and we use the notation of Section \ref{sec:lm_constructible}. Fix $\mathbf{r}=(r_1,r_2,\ldots,r_d)\in\mathbb{Z}^d$ and we compute the vectors of the canonical basis of height $2$ of $V(\Lambda_{\mathbf{r}})$. We set some notations for the $d$-symbols of height $2$. Let $0=i_0<i_1<\cdots<i_p=d$ such that for all $1 \leq k \leq p-1$ we have $r_{i_{k+1}}<r_{i_k}$ and for all $1 \leq k \leq p$ and $i_{k-1}< i \leq i_k$ we have $r_i=r_{i_k}$. By convention, we let $r_{0}=-\infty$ and $r_{d+1}=+\infty$. The $d$-symbols $S=(\beta_i)_{1\leq i \leq d}$ of height $2$ are the following:
\begin{itemize}
\item for $1 \leq i < j \leq d$, the $d$-symbol $S_{i,j}$ with $\beta_{i,r_i}=r_i+1$, $\beta_{j,r_j}=r_j+1$ and $\beta_{k,l}=l$ for all other values of $k$ and $l$,
\item for $1\leq i \leq d$, the $d$-symbol $S_{i}$ with $\beta_{i,r_i}=r_i+2$ and $\beta_{k,l}=l$ for all other values of $k$ and $l$,
\item for $1\leq i \leq d$, the $d$-symbol $S'_{i}$ with $\beta_{i,r_i}=r_i+1$, $\beta_{i,r_i-1}=r_i$ and $\beta_{k,l}=l$ for all other values of $k$ and $l$.
\end{itemize}
Among these symbols, the following are standard:
\begin{itemize}
\item for $1\leq k < l \leq p$ the $d$-symbol $S_{i_k,i_l}$ is standard,
\item for $1\leq k \leq p$ such that $i_k-i_{k-1}\geq 2$ the $d$-symbol $S_{i_k-1,i_k}$ is standard,
\item for $1\leq k \leq p$ the $d$-symbol $S_{i_k}$ is standard,
\item for $1\leq k \leq p$ such that $r_{i_k}-r_{i_{k+1}}\geq 2$ the $d$-symbol $S'_{i_k}$ is standard.
\end{itemize}
  
For these standard $d$-symbols, we now apply the algorithm of \cite{leclerc-toffin} and show that the  element $A_{\Sigma}$ of the intermediate basis already satisfies $A_\Sigma \equiv v_{\Sigma} \mod qF_{R}(\Lambda_{\mathbf{r}})$.

We denote by $\tilde{S}_i$ the symbol of height $1$ with $\beta_{i,r_i} = i+1$ and $\beta_{k,l}=l$ for all other values of $k,l$. If $1 \leq k \leq p$, we have
\begin{equation}
  F_{i_k}v_{S^0} = \sum_{i_{k-1}<i \leq i_k}q^{i_k-i}v_{\tilde{S}_i}.\label{eq:h1}\tag{$H_1$}  
\end{equation}

Let $1 \leq k < l \leq p$ and consider $\Sigma=S_{i_k,i_l}$. We obtain $A_{\Sigma}=F_{i_l}F_{i_k}v_{S^0}$. Then for $i_{k-1} < i \leq i_k$, the value of $F_{i_l}v_{\tilde{S}_i}$ depends on $r_{i_k}-r_{i_l}$
\[
  F_{i_l}v_{\tilde{S}_i}=
  \begin{dcases*}
    q^{i_{l-1}-i_l+1}v_{S'_i} + \sum_{i_{l-1}<j\leq i_l}q^{i_l-j}v_{S_{i,j}} & if $r_{i_k}=r_{i_l}+1$,\\
    \sum_{i_{l-1}<j\leq i_l}q^{i_l-j}v_{S_{i,j}}& otherwise.
  \end{dcases*}
\]
From the above formula and \eqref{eq:h1}, one obtains
\[
  A_{\Sigma}=
  \begin{dcases*}
    \sum_{i_{k-1}<i \leq i_k}q^{i_k-i}\left(q^{i_{l-1}-i_l+1}v_{S'_i} + \sum_{i_{l-1}<j\leq i_l}q^{i_l-j}v_{S_{i,j}}\right) & if $r_{i_k}=r_{i_l}+1$,\\
    \sum_{i_{k-1}<i \leq i_k}\sum_{i_{l-1}<j\leq i_l}q^{i_k-i+i_l-j}v_{S_{i,j}}& otherwise,
  \end{dcases*}
\]
and it is readily checked that $A_\Sigma \equiv v_{\Sigma} \mod qF_{R}(\Lambda_{\mathbf{r}})$. Note that 
if $r_{i_k}=r_{i_l}+1$ then $k=l+1$.

Let $1\leq k \leq p$ such that $i_k-i_{k-1}\geq 2$ and consider $\Sigma=S_{i_k-1,i_k}$. We obtain $(q+q^{-1})A_{\Sigma}=F_{i_k}^2v_{S^0}$. Then for $i_{k-1} < i \leq i_k$, we have 
\[
  F_{i_k}v_{\tilde{S}_i} = \sum_{i_{k-1}<j<i}q^{i_k-1-2}v_{S_{i,j}}+\sum_{i<j\leq i_k}q^{i_k-i}v_{S_{j,i}},
\]
so that from the above formula and \eqref{eq:h1}, one obtains
\[
  F_{i_k}^2v_{S^0} = (q+q^{-1})\sum_{i_{k-1}< i < j \leq i_k}q^{2i_k-i-j-1}v_{S_{i,j}}.
\]
It is readily checked that $A_\Sigma \equiv v_{\Sigma} \mod qF_{R}(\Lambda_{\mathbf{r}})$.

Let $1 \leq k \leq p$ and consider $\Sigma=S_{i_k}$. We obtain $A_{\Sigma} = F_{i_k+1}F_{i_k}v_{S^0}$. Then for $i_{k-1} < i \leq i_k$, the value of $F_{i_k+1}v_{\tilde{S}_i}$ depends on $r_{i_{k-1}}-r_{i_k}$:
\[
  F_{i_k+1}v_{\tilde{S}_i}=
  \begin{dcases*}
    v_{S_i}+\sum_{i_{k-2}<j\leq i_{k-1}}q^{i_{k-1}-j+1}v_{S_{i,j}} & if $r_{i_{k-1}}=r_{i_k}+1$,\\
    v_{S_{i}}& otherwise.
  \end{dcases*}
\]
Hence
\[
  A_{\Sigma}=
  \begin{dcases*}
    \sum_{i_{k-1}<i\leq i_k}q^{i_k-i}\left(v_{S_i}+\sum_{i_{k-2}<j\leq i_{k-1}}q^{i_{k-1}-j+1}v_{S_{i,j}}\right) & if $r_{i_{k-1}}=r_{i_k}+1$,\\
    \sum_{i_{k-1}<i\leq i_k}q^{i_k-i}v_{S_{i}}& otherwise,
  \end{dcases*}
\]
and we indeed have $A_\Sigma \equiv v_{\Sigma} \mod qF_{R}(\Lambda_{\mathbf{r}})$.

Finally, let $1\leq k \leq p$ such that $r_{i_k}-r_{i_{k+1}}\geq 2$ and consider $\Sigma = S'_{i_k}$. We obtain $A_{\Sigma}=F_{i_k-1}F_{i_k}v_{S^0}$. Then for $i_{k-1} < i \leq i_k$, we have
\[
  F_{i_k-1}v_{\tilde{S}_i} = v_{S'_i},
\]
since $r_{i_k}-r_{i_{k+1}}\geq 2$. Therefore
\[
  A_{\Sigma} = \sum_{i_{k-1} < i \leq i_k}q^{i_k-i}v_{S'_i},
\]
and $A_{\Sigma}\equiv v_{\Sigma} \mod qF_R(\Lambda_{\mathbf{r}})$.

From this, we obtain the Leclerc-Miyachi $\mathbf{r}$-constructible characters. The bijections between $d$-symbols of height $2$ and irreducible characters of $G(d,1,2)$ is given by
\[
  S_{i,j}\leftrightarrow \chi_{i,j},\quad S_i\leftrightarrow \chi_i\quad\text{and}\quad S'_i\leftrightarrow \chi'_i.
\]

\begin{prop}
  \label{prop:lm-cell-d12}
  For $1 \leq k < l \leq p$, the Leclerc-Miyachi $\mathbf{r}$-constructible character corresponding to the standard $d$-symbol $S_{i_k,i_l}$ is
  \[
    \gamma_{S_{i_k,i_l}}=
    \begin{dcases*}
      \sum_{i_{k-1}< i \leq i_k}\left(\chi'_i + \sum_{i_{l-1}<j \leq i_l}\chi_{i,j}\right) & if $r_{i_k}=r_{i_l}+1$,\\
      \sum_{i_{k-1}< i \leq i_k}\sum_{i_{l-1}<j \leq i_l}\chi_{i,j}& \text{otherwise}.
    \end{dcases*}
  \]

  For $1 \leq k \leq p$ such that $i_k-i_{k-1}\geq 2$, the Leclerc-Miyachi $\mathbf{r}$-constructible character corresponding to the standard $d$-symbol $S_{i_k-1,i_k}$ is
  \[
    \gamma_{S_{i_k-1,i_k}} = \sum_{i_{k-1}< i < j \leq i_k}\chi_{i,j}.
  \]

   For $1 \leq k \leq p$, the Leclerc-Miyachi $\mathbf{r}$-constructible character corresponding to the standard $d$-symbol $S_{i_k}$ is
  \[
    \gamma_{S_{i_k}} =
    \begin{dcases*}
      \sum_{i_{k-1}< i \leq i_k}\left(\chi_{i}+\sum_{i_{k-2}< j \leq i_{k-1}}\chi_{j,i}\right) & if $r_{i_{k-1}}=r_{i_k}+1$,\\
      \sum_{i_{k-1}<i\leq i_k}\chi_{i}& otherwise.
    \end{dcases*}
  \]

   For $1 \leq k \leq p$ such that $r_{i_k}-r_{i_{k+1}}\geq 2$, the Leclerc-Miyachi $\mathbf{r}$-constructible character corresponding to the standard $d$-symbol $S'_{i_k}$ is
  \[
    \gamma_{S'_{i_k}} = \sum_{i_{k-1} < i \leq i_k}\chi_i'.
  \]
\end{prop}




\section{A conjecture relating cellular and constructible characters}
\label{sec:conj}

We now state precisely the conjecture relating Calogero-Moser $\mathbf{c}$-cellular characters and Leclerc-Miyachi $\mathbf{r}$-constructible characters for the complex reflection group $G(d,1,n)$. Let $c\colon\Ref(G(d,1,n))\rightarrow \mathbb{C}$. We suppose that $c_0\neq 0$ and that for every $1 \leq i \leq d$ we have $k_i \in -\mathbb{N}c_0$. Finally, suppose also that
\[
  \frac{k^{\#}_1}{c_0} \leq \frac{k^{\#}_2}{c_0} \leq \cdots \leq \frac{k^{\#}_d}{c_0}.
\]

\begin{conj}
  \label{conj:cm-lm}
  Let $\mathbf{r} = -c_0^{-1}(k^{\#}_1,k^{\#}_2,\ldots,k^{\#}_d)$. Then the set of Calogero-Moser $\mathbf{c}$-cellular characters and the set of Leclerc-Miyachi $\mathbf{r}$-constructible characters coincide.
\end{conj}

If $d=2$ this conjecture is equivalent to the conjecture that Calogero-Moser $\mathbf{c}$-cellular characters for the Weyl group of type $B_n$ are Lusztig's constructible characters obtained via truncated induction \cite[Chapter 22]{lusztig-unequal}.

\begin{rem}
  If we start from a $d$-tuple $\mathbf{r}=(r_1,\ldots,r_d)$, one can choose a corresponding parameter $\mathbf{c}$ by $c_0\neq 0$ and $k^{\#}_i=-c_0r_i$.
\end{rem}

\begin{thm}
  \label{thm:proof-conj}
  If the parameter $\mathbf{c}$ is generic in the sense of Corollary \ref{cor:cm-generic} then Conjecture \ref{conj:cm-lm} is true.

  For the group $G(d,1,2)$ the Conjecture \ref{conj:cm-lm} is true for any $\mathbf{c}$.
\end{thm}

\begin{proof}
  With the change of parameters between $\mathbf{c}$ and $\mathbf{r}$, the generic case for the $\mathbf{c}$-cellular characters translates into the asymptotic case for the constructible characters. The result therefore follows from Corollary \ref{cor:cm-generic} and Corollary \ref{cor:lm-asymptotic}.

  For $G(d,1,2)$, we describe the equivalence relation $\sim$ introduced in Section \ref{sec:c0<>0}. Using the notation $(i_j)_{-1\leq j \leq p+1}$ introduced in Section \ref{sec:height2}, the equivalence classes of $\sim$ are the sets $\mathcal{O}_j=\{i_{j-1}+1,i_{j-1}+2,\ldots,i_j\}$ for $1 \leq j \leq p$.

  Using the explicit descriptions of the $\mathbf{c}$-cellular characters given in Proposition \ref{prop:cm-cell-d12-c0<>0} and of the $\mathbf{r}$-constructible characters given in Proposition \ref{prop:lm-cell-d12} we check that:
  \begin{itemize}
  \item for any $1\leq k < l\leq p$, the Leclerc-Miyachi $\mathbf{r}$-constructible character $\gamma_{S_{i_k,i_l}}$ is equal to the Calogero-Moser $\mathbf{c}$-cellular character $\gamma'_{\mathcal{O}_k}$ if $r_{i_k}=r_{i_l}+1$ and to the Calogero-Moser $\mathbf{c}$-cellular character $\gamma_{\mathcal{O}_k,\mathcal{O}_l}$ otherwise,
  \item for any $1\leq k \leq p$ such that $i_k - i_{k-1}\geq 2$, the Leclerc-Miyachi $\mathbf{r}$-constructible character $\gamma_{S_{i_k-1,i_k}}$ is equal to the Calogero-Moser $\mathbf{c}$-cellular character $\gamma_{\mathcal{O}_k,\mathcal{O}_k}$,
  \item for any $1\leq k < l\leq p$, the Leclerc-Miyachi $\mathbf{r}$-constructible character $\gamma_{S_{i_k}}$ is equal to the Calogero-Moser $\mathbf{c}$-cellular character $\gamma_{\mathcal{O}_k}$,
  \item for any $1\leq k < l\leq p$ such that $i_k-i_{k-1}\geq 2$, the Leclerc-Miyachi $\mathbf{r}$-constructible character $\gamma_{S'_{i_k}}$ is equal to the Calogero-Moser $\mathbf{c}$-cellular character $\gamma'_{\mathcal{O}_k}$.
  \end{itemize}
  It is easy to check that every Calogero-Moser $\mathbf{c}$-cellular character appears as a Leclerc-Miyachi $\mathbf{r}$-constructible character.
\end{proof}



\appendix

\section{Cellular characters}
\label{sec:appendix}

This appendix aims to define a general notion of cellular characters of a commutative algebra $A$, and is largely inspired from \cite[Appendix II]{bonnafe-thiel}. We fix $\Bbbk$ a field of characteristic $0$, $E$ a finite dimensional $\Bbbk$-vector space, $A$ a split subalgebra of $\End_{\Bbbk}(E)$ and $P$ an integral and integrally closed subalgebra with fraction field $K$.

Given $R$ a commutative $\Bbbk$-algebra, we denote by $RE$ (resp. $RA$) the extension of scalars $R\otimes_{\Bbbk} E$ (resp. $R\otimes_{\Bbbk} A$). Let $D_1,\ldots,D_n$ be some pairwise commuting elements of $\End_{PA}(PE)$. If $\mathfrak{p}$ is a prime ideal of $P$, we denote by $K_p(\mathfrak{p})$ the residue field at $\mathfrak{p}$ and by $D_i(\mathfrak{p})$ the image of $D_i$ in $\End_{P/\mathfrak{p}A}(P/\mathfrak{p}E)$. Finally, let $D=(D_1,\ldots,D_n)$ and $P[D]$ be the subalgebra of $\End_{PA}(PE)$ generated by $D_1,\ldots,D_n$.

We are interested in the decomposition of the vector space $KE$ as a $K[D]\otimes_{K} KA$-module, and more precisely of its class in the Grothendieck group $K_0(K[D]\otimes_{K} KA)$ of finite dimensional $K[D]\otimes_{K} KA$-modules.

Since the algebra $A$ is split, we obtain (\emph{cf.} \cite[Propositions 3.56 and 7.7]{curtis-reiner}) a bijection $\Irr(K[D])\times \Irr(KA) \rightarrow \Irr(K[D]\otimes_{K} KA)$ given by tensoring modules. This bijection induces an isomorphism of $\mathbb{Z}$-modules $K_0(K[D])\otimes_{\mathbb{Z}}K_0(KA)\rightarrow K_0(K[D]\otimes_{K} KA)$. We therefore decompose $[KE]$ in $K_0(K[D])\otimes_{\mathbb{Z}}K_0(KA)$ as follows:
\[
  [KE] = \sum_{L\in\Irr(K[D])}[L]\otimes \gamma_L^{P[D]},
\]
with $\gamma_L^{P[D]}\in K_0(KA)$. Since $A$ is split, we usually think of $\gamma_L^{P[D]}$ as an element of $K_0(A)$.

\begin{defn}
  \label{def:cellular}
  The set of cellular characters for $P[D]$ is the set of $\gamma_L^{P[D]}\in K_0(A)$ for $L$ running over the set of irreducible $K[D]$-modules.
\end{defn}

\begin{rem}
  It may happen that $\gamma_L^{P[D]} = \gamma_{L'}^{P[D]}$ for two non-isomorphic $K[D]$-modules.
\end{rem}

By extending the scalars to $P[\mathbf{X}]=P[X_1,\ldots,X_n]$ and by setting $\mathcal{D}=X_1D_1+\cdots + X_nD_n$, it is shown in \cite{bonnafe-thiel} that the set of cellular characters for $P[D]$ coincides with the set of cellular characters for $P[\mathbf{X}][\mathcal{D}]$. We therefore may and will suppose that $n=1$ and set $D=D_1$.

It is now easy to describe all the irreducible $K[D]$-modules. Denote by $\Pi$ the characteristic polynomial of $D$, which is a unital polynomial in $P[\mathbf{t}]$. We then decompose $\Pi$ into a product of irreducible unital polynomials in $K[\mathbf{t}]$
\[
  \Pi = \Pi_1^{n_1}\cdots \Pi_r^{n_r}.
\]
We also denote by $\Pi^{\mathrm{sem}}$ the product of the $\Pi_i$'s without multiplicity. Since $P$ is integrally closed, the polynomials $\Pi_i$ and $\Pi^{\mathrm{sem}}$ have their coefficients in $P$ and we set $\mathcal{L}_i = P[\mathbf{t}]/\langle\Pi_i\rangle$. The set of irreducible $K[D]$-modules are therefore the extensions to $K$ of the $P[D]$-modules $\mathcal{L}_i$:
\[
  \Irr(K[D]) = \{K\mathcal{L}_1,\ldots,K\mathcal{L}_r\}.
\]
The main advantage of working over $P$ is that we can easily reduce modulo a prime ideal $\mathfrak{p}$ of $P$. Let $\Delta$ be the discriminant of the polynomial $\Pi^{\mathrm{sem}}$, and denote by $\Delta(\mathfrak{p})$ its reduction modulo a prime ideal $\mathfrak{p}$ of $P$.

\begin{prop}
  \label{prop:sum_cell}
  Let $\mathfrak{p}$ be a prime ideal of $P$ such that $P/\mathfrak{p}$ is integrally closed. Then the cellular characters for $P/\mathfrak{p}[D(\mathfrak{p})]$ are sums of cellular characters for $P[D]$. If moreover $\Delta(\mathfrak{p}) \neq 0$ then the sets of cellular characters for $P/\mathfrak{p}[D(\mathfrak{p})]$ and for $P[D]$ coincide.
\end{prop}

\begin{proof}
  We start by decomposing the reduction $\Pi(\mathfrak{p})$ modulo $\mathfrak{p}$ into a product of irreducible polynomials with coefficients in $k_P(\mathfrak{p})$:
  \[
    \Pi_i(\mathfrak{p}) = \prod_{j=1}^{d_i}\pi_{i,j}^{e_{i,j}},
  \]
  where $\pi_{i,j}\in k_P(\mathfrak{p})[\mathbf{t}]$ is unital and irreducible, $e_{i,j}\in \mathbb{Z}_{>0}$ and $\pi_{i,j}\neq \pi_{i,j'}$ for $j\neq j'$. Since $P/\mathfrak{p}$ is integrally closed, the polynomials $\pi_{i,j}$ have their coefficients in $P/\mathfrak{p}$.

  For $1\leq i \leq r$ and $1 \leq j \leq d_i$, we denote by $\mathcal{L}_{i,j}$ the $P/\mathfrak{p}[D(\mathfrak{p})]$-module $(P/\mathfrak{p})[\mathbf{t}]/\langle\pi_{i,j}\rangle$. In the Grothendieck group of $K_P(\mathfrak{p})[D(\mathfrak{p})]$ we therefore have the following equality
  \begin{equation}
    [k_P(\mathfrak{p})\mathcal{L}_{i}] = \sum_{j=1}^{d_i}e_{i,j}[k_P(\mathfrak{p})\mathcal{L}_{i,j}].\label{eq:red_Li}
  \end{equation}

  Since $\Bbbk$ has characteristic $0$, an equality between elements of $K_0(K[D]\otimes_{K}KA)$ is equivalent to an equality between the corresponding characters, so that we can specialize modulo $\mathfrak{p}$ the equality
  \[
    [KE] = \sum_{i=1}^r[K\mathcal{L}_i]\otimes \gamma_{K\mathcal{L}_i}^{P[D]}
  \]
  into
  \[
    [k_P(\mathfrak{p})E] = \sum_{i=1}^r[k_P(\mathfrak{p})\mathcal{L}_i]\otimes \gamma_{K\mathcal{L}_i}^{P[D]}.
  \]
  Using \eqref{eq:red_Li}, we see that the cellular characters for $P/\mathfrak{p}[D(\mathfrak{p})]$ are sums of cellular characters for $P[D]$.

  If moreover $\Delta(\mathfrak{p})\neq 0$, then $e_{i,j}=1$ for all $1\leq i \leq r$ and $1 \leq j \leq d_i$ and $\pi_{i,j}=\pi_{l,m}$ if and only if $(i,j) = (l,m)$. Then
  \[
    [k_P(\mathfrak{p})E] = \sum_{i=1}^r\sum_{j=1}^{d_i}[k_P(\mathfrak{p})\mathcal{L}_{i,j}]\otimes \gamma_{i}^{P[D]}
  \]
  and the set $\left\{k_P(\mathfrak{p})\mathcal{L}_{i,j}\ \middle\vert\ 1\leq i \leq r, 1 \leq j \leq d_i\right\}$ is exactly the set of irreducible representations of $k_P(\mathfrak{p})[D(\mathfrak{p})]$, so that $\gamma_{k_P(\mathfrak{p})\mathcal{L}_{i,j}}^{P/\mathfrak{p}[D(\mathfrak{p})]}=\gamma_{K\mathcal{L}_i}^{P[D]}$ for all $1 \leq i \leq r$ and $1 \leq j \leq d_i$.
\end{proof}





\bibliographystyle{habbrv}
\bibliography{biblio}


\end{document}